\documentclass[a4paper,12pt]{amsart}

\usepackage[mathcal]{euscript}   
\usepackage{mathrsfs}

\usepackage[shortlabels]{enumitem}
\usepackage[matrix,arrow]{xy}
\usepackage{comment}
\usepackage{amssymb,enumitem,verbatim,stmaryrd,xcolor,microtype,graphicx,needspace}
\usepackage[T1]{fontenc}
\usepackage[utf8]{inputenc}
\usepackage[english]{babel} 

\usepackage[top=3.3cm,bottom=3.3cm,left=3.25cm,right=3.25cm]{geometry}
\usepackage[bookmarksdepth=2,linktoc=page,colorlinks,linkcolor={red!80!black},citecolor={red!80!black},urlcolor={blue!80!black}]{hyperref}

\usepackage{tikz}\usetikzlibrary{matrix,arrows,decorations.markings}
\usepackage{tikz-cd}
\usetikzlibrary{automata,positioning,decorations.pathreplacing}
\newcommand\cprime\textquotesingle 

\usepackage[figuresleft]{rotating}
\usepackage{changepage}

\usepackage{resizegather}
\usepackage{multicol}
\tikzset{->-/.style={decoration={markings,mark=at position #1 with {\color{black}\arrow{>}}},postaction={decorate,very thick}}}
\tikzstyle{vertex}=[circle, draw, inner sep=0pt, minimum size=6pt]
\newcommand{\vertex}{\node[vertex]}

\usepackage[mathcal]{euscript}                               
\usepackage{mathptmx}                                        
\usepackage[colorinlistoftodos,bordercolor=orange,backgroundcolor=orange!20,linecolor=orange,textsize=footnotesize]{todonotes} \makeatletter \providecommand \@dotsep{5} \def\listtodoname{List of Todos} \def\listoftodos{\@starttoc{tdo}\listtodoname} \makeatother 

\allowdisplaybreaks 

\usepackage{etoolbox}
\makeatletter
\patchcmd{\@startsection}{\@afterindenttrue}{\@afterindentfalse}{}{}             
\patchcmd{\part}{\bfseries}{\bfseries\LARGE}{}{}
\patchcmd{\section}{\scshape}{\bfseries}{}{}\renewcommand{\@secnumfont}{\bfseries} 
\patchcmd{\@settitle}{\uppercasenonmath\@title}{\large}{}{}
\patchcmd{\@setauthors}{\MakeUppercase}{}{}{}
\addto{\captionsenglish}{} 
\addto{\captionsenglish}{} 
\addto{\captionsenglish}{} 
\makeatother

\usepackage{fancyhdr}

\theoremstyle{plain}
\newtheorem{thm}{Theorem}[section]
\newtheorem{cor}[thm]{Corollary}
\newtheorem{lemma}[thm]{Lemma}
\newtheorem{prop}[thm]{Proposition}
\newtheorem{thmA}{Theorem}  

\newtheorem*{thm*}{Theorem}
\newtheorem*{lem*}{Lemma}
\newtheorem{conj}[thm]{Conjecture}

\theoremstyle{definition}
\newtheorem{df}[thm]{Definition}
\newtheorem{rem}[thm]{Remark}

\newtheorem*{df*}{Definition}
\newtheorem*{ex*}{Example}
\newtheorem*{rem*}{Remark}

\DeclareRobustCommand{\gobblefour}[5]{}    

\DeclareFontFamily{OT1}{pzc}{}                                
\DeclareFontShape{OT1}{pzc}{m}{it}{<-> s * [1.10] pzcmi7t}{}
\DeclareMathAlphabet{\mathpzc}{OT1}{pzc}{m}{it}
\DeclareSymbolFont{sfoperators}{OT1}{bch}{m}{n} \DeclareSymbolFontAlphabet{\mathsf}{sfoperators} \makeatletter\def\operator@font{\mathgroup\symsfoperators}\makeatother 
\DeclareSymbolFont{cmletters}{OML}{cmm}{m}{it}
\DeclareSymbolFont{cmsymbols}{OMS}{cmsy}{m}{n}
\DeclareSymbolFont{cmlargesymbols}{OMX}{cmex}{m}{n}
\DeclareMathSymbol{\myjmath}{\mathord}{cmletters}{"7C}     \let\jmath\myjmath 
\DeclareMathSymbol{\myamalg}{\mathbin}{cmsymbols}{"71}     
\DeclareMathSymbol{\mycoprod}{\mathop}{cmlargesymbols}{"60}\let\coprod\mycoprod
\DeclareMathSymbol{\myalpha}{\mathord}{cmletters}{"0B}     \let\alpha\myalpha 
\DeclareMathSymbol{\mybeta}{\mathord}{cmletters}{"0C}      \let\beta\mybeta
\DeclareMathSymbol{\mygamma}{\mathord}{cmletters}{"0D}     \let\gamma\mygamma
\DeclareMathSymbol{\mydelta}{\mathord}{cmletters}{"0E}     \let\delta\mydelta
\DeclareMathSymbol{\myepsilon}{\mathord}{cmletters}{"0F}   \let\epsilon\myepsilon
\DeclareMathSymbol{\myzeta}{\mathord}{cmletters}{"10}      \let\zeta\myzeta
\DeclareMathSymbol{\myeta}{\mathord}{cmletters}{"11}       \let\eta\myeta
\DeclareMathSymbol{\mytheta}{\mathord}{cmletters}{"12}     \let\theta\mytheta
\DeclareMathSymbol{\myiota}{\mathord}{cmletters}{"13}      \let\iota\myiota
\DeclareMathSymbol{\mykappa}{\mathord}{cmletters}{"14}     \let\kappa\mykappa
\DeclareMathSymbol{\mylambda}{\mathord}{cmletters}{"15}    \let\lambda\mylambda
\DeclareMathSymbol{\mymu}{\mathord}{cmletters}{"16}        \let\mu\mymu
\DeclareMathSymbol{\mynu}{\mathord}{cmletters}{"17}        \let\nu\mynu
\DeclareMathSymbol{\myxi}{\mathord}{cmletters}{"18}        \let\xi\myxi
\DeclareMathSymbol{\mypi}{\mathord}{cmletters}{"19}        \let\pi\mypi
\DeclareMathSymbol{\myrho}{\mathord}{cmletters}{"1A}       \let\rho\myrho
\DeclareMathSymbol{\mysigma}{\mathord}{cmletters}{"1B}     \let\sigma\mysigma
\DeclareMathSymbol{\mytau}{\mathord}{cmletters}{"1C}       \let\tau\mytau
\DeclareMathSymbol{\myupsilon}{\mathord}{cmletters}{"1D}   \let\upsilon\myupsilon
\DeclareMathSymbol{\myphi}{\mathord}{cmletters}{"1E}       \let\phi\myphi
\DeclareMathSymbol{\mychi}{\mathord}{cmletters}{"1F}       \let\chi\mychi
\DeclareMathSymbol{\mypsi}{\mathord}{cmletters}{"20}       \let\psi\mypsi
\DeclareMathSymbol{\myomega}{\mathord}{cmletters}{"21}     \let\omega\myomega
\DeclareMathSymbol{\myvarepsilon}{\mathord}{cmletters}{"22}\let\varepsilon\myvarepsilon
\DeclareMathSymbol{\myvartheta}{\mathord}{cmletters}{"23}  \let\vartheta\myvartheta
\DeclareMathSymbol{\myvarpi}{\mathord}{cmletters}{"24}     \let\varpi\myvarpi
\DeclareMathSymbol{\myvarrho}{\mathord}{cmletters}{"25}    \let\varrho\myvarrho
\DeclareMathSymbol{\myvarsigma}{\mathord}{cmletters}{"26}  \let\varsigma\myvarsigma
\DeclareMathSymbol{\myvarphi}{\mathord}{cmletters}{"27}    \let\varphi\myvarphi

\DeclareMathOperator{\GL}{GL}

\newcommand\A{{\mathbb A}}

\newcommand\C{{\mathbb C}}

\newcommand\FF{{\mathbb F}}

\newcommand\N{{\mathbb N}}

\renewcommand\P{{\mathbb P}}

\newcommand\Z{{\mathbb Z}}

\newcommand\cA{{\mathcal A}}

\newcommand\cE{{\mathcal E}}

\newcommand\cG{{\mathcal G}}

\newcommand\cH{{\mathcal H}}

\newcommand\cN{{\mathcal N}}
\newcommand\cO{{\mathcal O}}

\newcommand\cR{{\mathcal R}}

\newcommand\cV{{\mathcal V}}

\newcommand{\Fq}{\mathbb{F}_q}

\newcommand\id{\textup{id}}
\renewcommand\geq{\geqslant}
\renewcommand\leq{\leqslant}

\newcommand{\norm}[1]{|#1|}

\renewcommand\emptyset\varnothing

\title{ Hecke eigenspaces for the projective line }

\author{Roberto Alvarenga}
\address{\rm Roberto Alvarenga, São Paulo State University (UNESP), Campus São José do Rio Preto, Brazil.}
\email{roberto.alvarenga@unesp.br}

\author{Nans Bonnel}
\address{\rm Nans Bonnel, Ecole Normale Superieur Paris-Saclay, France.}
\email{nans.bonnel@ens-paris-saclay.fr}

\allowdisplaybreaks

\begin{document}

\maketitle

\begin{abstract} In this article we investigate the action of (ramified and unramified) Hecke operators on automorphic forms for the function field of the projective line defined over $\Fq$ and for the group $\GL_2$. We first compute the dimension of the Hecke eigenspaces for every generator of the unramified Hecke algebra. Thus, we consider the ramification in a point of degree one and explicitly describe  the action of certain ramified Hecke operators on automorphic forms. Moreover, we also compute the dimensions of its eigenspaces for those ramified Hecke operators. We finish the article considering more general ramifications, namely  those one attached to a closed point of higher degree. 
\end{abstract} 

\tableofcontents

\section{Introduction}


 Hecke operators and its action on automorphic forms play a key role in the context of modern number theory. For instance, the Modularity Theorem  states that there is an automorphic form attached to each rational elliptic curve, moreover, that this automorphic form is actually a Hecke eigenform.  Also known as Taniyama–Shimura-Weil conjecture, the Modularity Theorem was proved by Wiles in \cite{wiles} (with a key step given by joint work with Taylor \cite{taylor-wiles-95}) for semi-stable elliptic curves, completing the proof of the Fermat Last Theorem after more than 350 years. The Modularity Theorem was  completely proved by Breuil, Conrad, Diamond and Taylor in \cite{bcdt-01}. More generally, the Hecke operators and its action on the space of automorphic forms play a central role in the geometric Langlands correspondence, proved by Drinfeld and Lafforgue  for $\GL_n$ in  \cite{drinfeld83} and \cite{lafforgue-02}. 

Motivated by previous examples, this work is concerned with the action of Hecke operators on automorphic forms. Namely, we consider the situation where these objects are defined over the projective line attached to a finite field and for $\GL_2$. We actually consider both ramified and unramified Hecke operators, describe its action on the space of automorphic forms by building its graphs of Hecke operators,  and use that to calculate the dimension of Hecke eigenforms space.

Followingly, we give precise definitions for the main objects that are considered throughout the article and state our main results. 
Let $q$ be a prime power and $\Fq$ be the finite field with $q$ elements. Even though from the next section onward we shall be interested in the case where $X = \mathbb P^1$ defined over $\Fq$ (and thus $F = \Fq(t)$), we introduce in this section the main objects of this work for any $X$ a geometrically irreducible  smooth projective curve defined over $\Fq$ and $F$ its function field. Let $g$ stand for the genus of $X$ and $|X|$ be the set of closed points of $X$ or, equivalently, the set of places in $F$.
For $x \in |X|$, we denote by $F_x$ the completion of $F$ at $x$, by $\mathcal{O}_x$ the ring of integers of $F_x$, by $\pi_x \in \mathcal{O}_x$ (we can assume $\pi_x \in F$) a uniformizer of $x$ and by $q_x$ the cardinal of the residue field $\kappa(x):=\mathcal{O}_x/(\pi_x) \cong \mathbb{F}_{q_x}.$ Moreover, we denote by $|x|$  the degree of $x$, which is defined by the extension field degree $[\kappa(x) : \Fq]$. In other words, $q_x = q^{|x|}$. 
Let $| \cdot |_x$ the absolute value of $F_x$ (resp. $F$) such that $|\pi_x|_x = q_{x}^{-1},$ we call $| \cdot |_x$ the local norm for each $x \in |X|$.

Let $\mathbb{A}$ be the adele ring of $F$ and $\mathbb{A}^{\times} $ the idele group. We denote $\mathcal{O}_{\mathbb{A}} := \prod \mathcal{O}_x $ where the product is
taken over all places $x$ of $F$. We might assume $F_x$ being embedded into the adele ring $\mathbb{A}$ by sending an element $a \in F_x$ to the adele $(a_y)_{y \in |X|}$ with $a_x = a$ and $a_y = 0$ for $y \neq x$. Let $G(\mathbb{A}):= \mathrm{GL}_n(\mathbb{A})$, $Z$ be the center of  $G(\mathbb{A})$, $G(F):= \mathrm{GL}_n(F)$,  and $K:= \mathrm{GL}_n(\mathcal{O}_{\mathbb{A}})$ the standard maximal compact open subgroup of $G(\mathbb{A})$. For $x \in |X|$, consider $G_x=\GL_n(F_x)$ and $Z_{x}$ the center of $G_x$. Note that $G(\mathbb{A})$ comes together with the adelic topology that turns $G(\mathbb{A})$ into a locally compact group. Hence $G(\mathbb{A})$ is endowed with a Haar measure. We fix the Haar measure on $G(\mathbb{A})$ for which $\mathrm{vol}(K)=1.$ The topology of $G(\mathbb{A})$ has a neighborhood basis $\mathcal{V}$ of the identity matrix that is given by all subgroups
$$K' = \prod_{x \in |X|} K_{x}' < \prod_{x \in |X|}K_x = K$$
where $K_x := \mathrm{GL}_n(\mathcal{O}_x)$, such that for all $x \in |X|$ the subgroup $K_{x}'$ of $K_x$ is open and consequently of finite index and such that $K_{x}^{'}$ differs from $K_x$ only for a finite number of places.

\subsection*{The Hecke algebra} Let $C^0(G(\mathbb{A}))$ be the space of continuous and $\C$-valuated functions on $G(\mathbb{A})$. A function in $C^0(G(\mathbb{A}))$ is called smooth if it is locally constant. We might now define one of the main objects of this work.

\begin{df} \label{def-hecke} The complex vector space $\mathcal{H}$ of all smooth compactly supported functions $\Phi : G(\mathbb{A}) \rightarrow \C$ together with the convolution product
$$\Phi_1 \ast \Phi_2: g \longmapsto \int_{G(\mathbb{A})} \Phi_1(gh^{-1})\Phi_2(h)dh$$
for $\Phi_1, \Phi_2 \in \mathcal{H}$ is called the Hecke algebra for $G(\mathbb{A})$. Its elements are called Hecke operators.
\end{df}

The zero element of $\mathcal{H}$ is the zero function, but there is no multiplicative unit. 
For $K' \in \mathcal{V},$ we define $\mathcal{H}_{K'}$ to be the subalgebra of all (left and right) bi-$K'$-invariant elements. These subalgebras have multiplicative units. Namely, the normalized characteristic function $\epsilon_{K'} := (\mathrm{vol}K')^{-1} \mathrm{char}_{K'}$ acts as the identity on $\mathcal{H}_{K'}$ by convolution. 

\begin{df} When $K'=K$ above, we call $\mathcal{H}_K$ the unramified (or spherical) part of $\mathcal{H}$, and its elements are called unramified (or spherical) Hecke operators. For $K' \in \mathcal{V},\ K'\neq K$, $\mathcal{H}_{K'}$ is called the ramified part of $\mathcal{H}$, and its elements are called   ramified Hecke operators.
\end{df}

It is well known that every $\Phi \in \mathcal{H}$ is bi-$K'$-invariant for some $K' \in \mathcal{V}$, cf.\ \cite[Prop. 1.4.4]{oliver-thesis}. In particular, 
$$\mathcal{H} = \bigcup_{K' \in \mathcal{V}} \mathcal{H}_{K'}.$$

\subsection*{Automorphic Forms} The group $G(\mathbb{A})$ acts on $C^0(G(\mathbb{A}))$ through the right regular representation
$$\rho : G(\mathbb{A}) \rightarrow \mathrm{Aut}(C^0(G(\mathbb{A}))),$$ 
which is defined by right translation of the argument: $(g.f)(h):= (\rho(g)f)(h) := f(hg)$ for $g,h \in G(\mathbb{A})$ and $f \in C^0(G(\mathbb{A})).$ Hence we have the induced action of the Hecke algebra $\mathcal{H}$  on $C^{0}(G(\mathbb{A}))$ given by
$$\Phi(f): g \longmapsto \int_{G(\mathbb{A})} \Phi(h) f(gh) dh.$$

We say that a function $f \in C^{0}(G(\mathbb{A}))$ is $\mathcal{H}$-finite if the space $\mathcal{H}\cdot f$ is finite dimensional. 
A function $f \in C^{0}(G(\mathbb{A}))$ is called $K$-finite if the complex vector space that is generated by $\{k.f\}_{k \in K}$ is finite dimensional. 

We embed $G(\mathbb{A}) \hookrightarrow \mathbb{A}^{n^2 +1}$ via $g \mapsto (g, \det(g)^{-1}).$ We define a local height $\norm{ g_x}_x$ on $G(F_x) :=\mathrm{GL}_n(F_x)$ by 
restricting the height function 
$$(v_1, \ldots, v_{n^2+1}) \mapsto \mathrm{max}\{|v_1|_x, \ldots, |v_{n^2+1}|_x\}$$ 
on $F_{x}^{n^2+1}$. We note that $\norm{ g_x }_x \geq 1$ and that $\norm{ g_x}_x =1$ if $g_x \in K_x.$ We define the global height $\norm{ g}$ to be the product of the local heights. We say that $f \in C^{0}(G(\mathbb{A}))$ is of moderate growth if there exists constants $C$ and $N$ such that
$$|f(g)|_{\C} \leq C \norm{ g}^{N}$$  
for all $g \in G(\mathbb{A}).$

\begin{df} The space of automorphic forms $\mathcal{A}$ (with trivial central character) is the complex vector space of all functions $f \in C^{0}(G(\mathbb{A}))$, which are smooth, $K$-finite, of moderate growth,  left $G(F)$-invariant and $\mathcal{H}$-finite. Its elements are called automorphic forms. 
\end{df}

For $K' \in \cV$ and every subspace $V \subseteq \mathcal{A},$ let $V^{K'}$ be the subspace of all $f \in V$ that are right $K'$-invariant.  The above action of $\mathcal{H}$ on $C^{0}(G(\mathbb{A}))$ restricts to an action of $\mathcal{H}_{K'}$ on $V^{K'}$.  
Since every automorphic form $f \in \mathcal{A}$ is  right $K'$-invariant for some $K' \in \mathcal{V}$ (cf.\ \cite[Prop. 1.3.7]{oliver-thesis}), $\mathcal{A} = \bigcup_{K' \in \mathcal{V}} \mathcal{A}^{K'}$. 
Hence, functions on $\mathcal{A}^{K'}$ (i.e.\ every automorphic form) can be identified with functions on the double coset $G(F) \setminus G(\mathbb{A})/Z K'$ 
that are of moderate growth.
The space $\cA^{K}$ is known as the space of unramified automorphic forms. If $f \in \cA^{K'}$ for $K' \neq K$, we say that $f$ has some ramification.

\subsection*{The main results}
Given $K' \in \cV$, this work is devoted to a better understanding of the action of $\cH_{K'}$ on $\cA^{K'}$. For instance,  investigate the 
$\cH_{K'}$-eigenspace given by the set of $f \in \mathcal{A}^{K'}$, such that there exists an eigencharacter $\lambda_f$ satisfying 
$\Phi f = \lambda_f(\Phi) f$ for all $\Phi \in \cH_{K'}$. In this case, we say $f$ is an $\mathcal{H}_{K'}$-eigenform with eigencharacter $\lambda_f$. Note that $\lambda_f$ defines a morphism of $\C$-algebras from $\mathcal{H}_{K'}$ to $\C$. Hence, $\lambda_f$ indeed defines an additive character on $\mathcal{H}_{K'}.$ 

Given $\Phi \in \cH_{K'}$ and $\lambda \in \C$, we shall be interested in getting information about the dimension of the Hecke $\Phi$-eigenspace 
\[ \cA^{K'}(\Phi, \lambda):= \big\{f \in \cA^{K'} \;\big|\; \Phi(f)=\lambda f \big\}.\] 
If $f \in \cA^{K'}(\Phi, \lambda)$, we say that $f$ is an $\Phi$-eigenform. Observe that if $f \in \mathcal{A}^{K'}$ is an $\mathcal{H}_{K'}$-eigenform with eigencharacter $\lambda_f$, then $f$ is an $\Phi$-eigenform, for every $\Phi \in \cH_{K'}$ and $\lambda_f(\Phi)=\lambda$.

Our first main result, Theorem \ref{thm-mainunramified}, concerns about the unramified space of Hecke eigenforms and can be stated as follows. 

\begin{thmA} \label{thmA} Let $y \in \P^1$ be a closed point of degree $d$. Let $\Phi_y$ stand for the unramified Hecke operator given by the characteristic function of $K\big(\begin{smallmatrix}
  \pi_y & 0\\
  0 & 1
\end{smallmatrix}\big)K$,
where $\pi_y$ is an uniformizer of $y$. Then $\dim \cA^{K}(\Phi_y, \lambda)=d$, for every $\lambda \in \C^{\times}.$
    \end{thmA}

In the sequence, we fix $x \in \P^1$ a closed point of degree one and consider $K' < K$ with ramification at $x$ given by
\[ K' := K_x'\times \prod_{y\neq x}K_y \quad \text{ such that } \quad K_x':=\big\{ k \in K_x \;|\; k \equiv \id~\textrm{mod}(\pi_x)\big\}.\]
If $y \in \P^1$ is any a closed point, let  $\Phi_y'$ stand for the ramified Hecke operator given by the characteristic function of $K' \big( \begin{smallmatrix}
    \pi_y & 0  \\
     0 & 1
\end{smallmatrix} \big) K'.$ 
Hence, we obtain the ramified versions of the Theorem \ref{thmA}.

\begin{thmA} \label{thmB} If $\lambda \in \C^{\times}$, then $\dim \cA^{K'}(\Phi_x', \lambda)=1.$ 
\end{thmA}

\begin{thmA} \label{thmC} Let $y \in \P^1$ be  a closed point, $y\neq x$. If $\lambda \in \C^{\times}$, then $\dim \cA^{K'}(\Phi_y', \lambda)=d(q+1).$
\end{thmA}

These are Theorem \ref{thm-mainramified1} and Theorem  \ref{thm-mainramified2}, respectively. Moreover, as a corollary of Theorem \ref{thmB}, we have the following. 

\begin{thmA} \label{thmD} There are no nontrivial ramified (with respect to $K'$)  cuspidal automorphic forms  for the projective line.
\end{thmA}

This is Corollary \ref{cor-trivicuspramifiedforms}. We refer to Definition \ref{def-cusp} for the precise definition of cusp (automorphic) forms.

The above $K'$, as well as Theorems \ref{thmB}, \ref{thmC} and \ref{thmD} are considered in section \ref{sec-ramified}. 
We end the article by considering the generalization of the ramification given by $K'$ in section \ref{sec-generalramification}. Regarding this new ramification, in addition to some technical results, we provide a theoretical description of the action of ramified Hecke operators on the space of ramified automorphic forms. This is stated in Theorem \ref{thm-generalramifiedgraphs}. Moreover, we announce a conjecture for the dimension of the space of ramified  Hecke eigenforms.


\section{Graphs of Hecke operators}
The graphs of Hecke operators play a central role in this work.  We shall explicitly describe  the action of $\mathcal{H}_{K'}$ on $\cA^{K'}$ for both unramified and ramified cases with these graphs. This allows us to obtain a complete description of the Hecke eigenspaces and, therefore, prove our main results. Hence, we dedicated this section to a careful introduction of this object. 

The following proposition is an important result that leads to the definition of the graphs of Hecke operators. 
 
\begin{prop}{\cite[Prop. 1.3]{oliver-graphs}} \label{propfund}  Let $K' \in \mathcal{V}$ and fix $\Phi \in \mathcal{H}_{K'}$. For all  classes of adelic matrices $[g] \in G(F)Z \setminus G(\mathbb{A}) / K',$ there is a unique set of pairwise distinct classes $[g_1], \ldots,[g_r] \in G(F) Z\setminus G(\mathbb{A}) / K'$ and numbers $m_1,  \ldots, m_r \in \C^{*}$ such that, for all $f \in \mathcal{A}^{K'}$, we have  
$$\Phi(f)(g) = \sum_{i=1}^{r} m_i f(g_i).$$
\end{prop}

Above proposition allow us to define the graph of a Hecke operator as follows.

\begin{df} \label{defgraphs} For $[g],[g_1], \ldots,[g_r] \in G(F) \setminus G(\mathbb{A}) /Z K'$ as in the Proposition \ref{propfund}, we denote  
$\mathcal{V}_{\Phi,K'}([g]) := \{([g],[g_i],m_i)\}_{i=1, \ldots, r}$. Thus, we define the graph $\mathcal{G}_{\Phi,K'}$ of $\Phi$ relative to $K'$ whose vertices are 
\[\mathrm{Vert}\; \mathcal{G}_{\Phi,K'} = G(F) \setminus G(\mathbb{A})/Z K'\]
and the oriented weighted edges 
\[\mathrm{Edge} \; \mathcal{G}_{\Phi,K'} = \bigcup_{[g] \in \mathrm{Vert} \mathcal{G}_{\Phi,K'}} \mathcal{V}_{\Phi,K'}([g]).\]
The classes $[g_i]$ are called the $\Phi$-neighbors of $[g]$ (relative to $K'$). Moreover, 
for the triples $([g],[g_i],m_i)$, we call the first entry $[g]$ by  origin, the second entry $[g_i]$ by terminus and the third $m_i$ by the multiplicity of the edge. 

\end{df}

These graphs are introduced by Lorscheid in \cite{oliver-graphs} for $\GL_2$ and generalized  in \cite{alvarenga19} to $\GL_n$. In \cite{oliver-elliptic} and \cite{alvarenga20} these graphs are described when $X$ is an elliptic curve. Moreover, in \cite{oliver-toroidal} Lorscheid applies this theory to probe the space of toroidal automorphic forms. 

\begin{rem}
For better readability, even though a vertex in the graph of a Hecke operator is a class $[g]$ in  $G(F) \setminus G(\mathbb{A}) /Z K'$, we shall denote it just by $g$. Moreover, we will denote $m_{\Phi, K'}(g,h)$ the multiplicity of the edge between $g$ and $h$ in $\mathcal{G}_{\Phi,K'}$. Observe that $m_{\Phi, K'}(g,h)= 0$ if $g$ and $h$ are not $\Phi$-neighbors.
\end{rem}

By Proposition \ref{propfund} and the definition of the graph of $\Phi \in \cH_{K'}$, we have for $f \in \mathcal{A}^{K'}$ and $[g] \in G(F) \setminus G(\mathbb{A}) /Z K'$ that
$$\Phi(f)(g) = \displaystyle\sum_{\substack{(g,g_i,m_i)\\ \in \mathrm{Edge}\mathcal{G}_{\Phi,K'}}} m_i f(g_i) = 
\displaystyle\sum_{h \in \mathrm{Vert}\; \mathcal{G}_{\Phi,K'}} m_{\Phi,K'}(g,h) f(h).$$

Since $\mathcal{H} = \bigcup \mathcal{H}_{K'}$, with $K'$ running over all compact opens in $G(\mathbb{A})$, the notion of the graph of a Hecke operator applies to any $\Phi \in \mathcal{H}.$
By construction, the set of vertices of the graph of a Hecke operator $\Phi \in \mathcal{H}_{K'}$  only depends on $K'$, whereas the edges depend on the particular choice of $\Phi.$ There are at most one edge for each two vertices, and the weight of an edge is always non-zero. Each vertex is connected to only a finite number of other vertices.


\section{Unramified Hecke eigenspaces}
\label{sec-unramified}

Through the article, empty entries in the matrices represents  zero entries. Moreover, for $x \in |X|$, we consider $ g \in G_x$ as a matrix in $G(\A)$ by completing the adele with the identity matrix in the places different from $x$. If $g \in G_x \cap G_y$, for $x,y \in |X|$, to avoid ambiguity we write $(g)_x \in G(\mathbb{A})$ for the adelic matrix given by completing the adele with the identity matrix in the places different of $x$. Analogously, $(g)_y$ is adelic matrix given by completing the adele with the identity matrix in the places that are different from $y$.

\subsection*{The unramified Hecke algebra} In this section, we will be interested in the unramified Hecke algebra $\cH_K$ and its action on 
$\cA^K$. For $x \in |X|$, let $\Phi_x$ be the characteristic function of 
\[K\begin{pmatrix}
    \pi_x &  \\
     & 1
\end{pmatrix} K\] 
and $\Phi_{x,0}$ be the characteristic function of 
\[K\begin{pmatrix}
    \pi_x &  \\
     & \pi_x
\end{pmatrix} K = 
\begin{pmatrix}
    \pi_x &  \\
     & \pi_x
\end{pmatrix} K. \]
 Both $\Phi_x$ and $\Phi_{x,0}$ are elements of $\cH_K$. A theorem of Satake, cf.\ \cite[Thm. 4.6.1 and Prop. 4.6.1]{bump} or \cite{satake63}, state that these Hecke operators generates the unramified Hecke algebra. Namely, identifying the characteristic function on $K$ with $1 \in \C$ yields $\cH_K \cong \C[\Phi_x, \Phi_{x,0}, \Phi_{x,0}^{-1}]_{x \in |X|}$. The action of $\cH_K$ in $\cA^K$ is completely determined by the action of $\Phi_x$ for $x \in |X|$, cf.\ \cite[Sec. 1]{oliver-graphs}. This justifies our reduction in the following.

\subsection*{The projective line case} From now on, we consider the case where $X = \mathbb P^1$, i.e.\ $F = \Fq(t)$, and $G = \GL_2$. Moreover, we fix $x \in |\P^1|$ a place of degree~$1$ such that $t=\pi_x$. 
The main result of this section is the calculus of dimension of $\cA^{K}(\Phi_y, \lambda)$, the eigenspaces of $\Phi_y \in \cH_K$, for every $y \in |\P^1|$ and $\lambda \in \C^{\times}$. We observe that for $|y|=1$ this dimension is $1$ and was already calculated in \cite{oliver-thesis}.

For every $n \in \Z$ and for every place $y$ in $\P^1$, we define 
 \[ p_{ny}:=\begin{pmatrix} 
         \pi_y^{-n}& \\ &1 
   \end{pmatrix} \in G_y := G(F_y).\]
   We will denote $p_y$ for $p_{1y}$ and $p_0$ for $p_{0y}$.

\begin{prop}{\cite[App. A]{oliver-graphs}} Let $x \in |\P^1|$ be the fixed place of degree~$1$. Then, the set 
$\big\{p_{nx} \;|\; n\in \N\big\}$ is a system of representatives for the double coset $G(F)\setminus G(\mathbb{A})/Z K$. Hence, denoting $c_{nx}$ the class of $p_{nx}$ in $G(F) \setminus G(\mathbb{A})/Z K,$ yields 
 \[\mathrm{Vert}\; \mathcal{G}_{\Phi,K} = \{c_{nx}\}_{n \geq0}\]
 for every $\Phi \in \cH_K.$
$\hfill \square$    
\end{prop}

Let $\Gamma_x:=\bigcap_{y\neq x} (G(F)\cap K_y) =\GL_2(\FF_q[\pi_x^{-1}])$.  The map 
 \[ (.)_x : \Gamma_x \setminus G_x/Z_xK_x \longrightarrow G(F) \setminus G(\mathbb{A})/Z K\]
which is defined by completing the adele with the identity matrix in the places different from $x$, is  well-defined and bijective. 
This is the Strong Approximation for $\mathrm{SL}_2$, we refer \cite[Lemma 9.5.9 and Thm. E.2.1]{Laumon97} for a proof.
Although, in the next section, we prove it for every $K' \in \mathcal{V}$.

Let $y \in |\P^1|$ of degree $d$. As we pointed out in the previous section, in order to describe the action of $\Phi_y$ on $\cA^K$ we  build the graph $\cG_{\Phi_y,K}$. We shall do it using the equivalence between $\Gamma_x \setminus G_x/Z_xK_x$ and $G(F) \setminus G(\mathbb{A})/Z K$. We denote $g \sim g'$ whether $g,g' \in G_x$ are equivalent in $\Gamma_x \setminus G_x/Z_xK_x$. Follows from \cite[App. A]{oliver-graphs} that
for every $n \in \N$, the neighbors of $c_{nx}$ in $\cG_{\Phi_y,K}$ are the classes of
\[  \begin{pmatrix} \pi_x^d&b_0+b_1\pi_x+...+b_{d-1}\pi_x^{d-1} \\ &1 \end{pmatrix}p_{nx}=\begin{pmatrix} \pi_x^{d-n}&b_0+b_1\pi_x+...+b_{d-1}\pi_x^{d-1} \\ &1 \end{pmatrix} \]
and 
\[ \begin{pmatrix} 1& \\ &\pi_x^d \end{pmatrix}p_{nx}=\begin{pmatrix} \pi_x^{-n}& \\ &\pi_x^d \end{pmatrix}
\]
in $\Gamma_x \setminus G_x/Z_xK_x$, for  $b_0,b_1,...,b_{d-1} \in \FF_q$. See  \cite[Prop. 5.4]{alvarenga19} for generalization to higher rank matrices. 

\begin{df} For every $k \in \N$, we define $\cN_k:=\{c_0,c_x,c_{2x},...,c_{kx}\} \subset \mathrm{Vert}\; \cG_{\Phi_y,K}$. We call $\mathfrak{N}_d := \cN_{d-1}$ the nucleus of $\cG_{\Phi_y,K}$, where $d = |y|$. See \ \cite[Def. 8.6]{oliver-graphs}. 
\end{df}


\begin{lemma} \label{lemma2.1}
    Let $b_0, b_1,...,b_{d-1} \in \FF_q$ and $n \in \N$ be such that $ n\leq d-1$. Then the class of 
    \[\begin{pmatrix} \pi_x^{d-n}&b_0+b_1\pi_x+...+b_{d-1}\pi_x^{d-1} \\ &1 \end{pmatrix}\]
    in  $\Gamma_x \setminus G_x/Z_xK_x$ belongs to $\mathfrak{N}_d $ unless that $n=0$ and $b_i = 0$
    for $i=1, \ldots, d-1$.
\end{lemma}

\begin{proof} Observe that if $n=0$ and $b_1 = \cdots = b_{d-1}=0$,
we get the equivalence class $c_{dx}$. 

In the following we suppose $n>0$ and some $b_i \neq 0$ for $i \in \{1, \ldots, d-1\}.$ We denote $m_1=d-n\geq1$ and observe that
\begin{align*}
\begin{pmatrix} \pi_x^{m_1}&b_0+b_1\pi_x+...+b_{d-1}\pi_x^{d-1} \\ &1 \end{pmatrix} &=
\begin{pmatrix} 1&b_0 \\ &1 \end{pmatrix}
\begin{pmatrix} \pi_x^{m_1}&\sum_{i=1}^{m_1-1}b_i\pi_x^{i} \\ &1 \end{pmatrix}
\begin{pmatrix} 1 &\sum_{i=m_1}^{d-1}  b_i\pi_x^{i-m_1} \\ &1 \end{pmatrix}\\
&\sim \begin{pmatrix} \pi_x^{m_1}&\sum_{i=1}^{m_1-1}b_i\pi_x^{i} \\ &1 \end{pmatrix}.
\end{align*}
Let $k_1$ stand for the smallest index $i \in \{1,\ldots, d-1\}$, such that $b_i \neq 0$. If $k_1 > m_1-1$,  we can conclude from above. 

Otherwise, we can consider 
$s:=b_{k_1}+...+b_{m_1-1}\pi_x^{m_1-1-k_1} \in \cO_x^{\times}$, $s^{-1}=:\sum_{i=0}^{+\infty}a_i\pi_x^{i}.$
and conclude that
 \begin{align*}
\begin{pmatrix} \pi_x^{m_1}&\sum_{i=1}^{m_1-1}b_i\pi_x^{i} \\ &1 \end{pmatrix}
&=\begin{pmatrix} \pi_x^{m_1}&s\pi_x^{k_1} \\ &1 \end{pmatrix}\\
&=\begin{pmatrix} &1 \\ 1& \end{pmatrix}
\begin{pmatrix} \pi_x^{m_1-2k_1}&s^{-1}\pi_x^{-k_1} \\ &1 \end{pmatrix}
\begin{pmatrix} s^{-1}\pi_x^{k_1}& \\ &s^{-1}\pi_x^{k_1} \end{pmatrix}
\begin{pmatrix} -1& \\ s\pi_x^{m_1-k_1}&s^2 \end{pmatrix}\\
&\sim \begin{pmatrix} \pi_x^{m_1-2k_1}&s^{-1}\pi_x^{-k_1} \\ &1 \end{pmatrix}\\
&= \begin{pmatrix} \pi_x^{m_1-2k_1}&\sum_{i=0}^{+\infty}a_i\pi_x^{i-k_1} \\ &1 \end{pmatrix} =: M.\\
\end{align*}
Hence, 
\[
M =     \begin{pmatrix}1&\sum_{i=0}^{k_1}a_i\pi_x^{i-k_1} \\ &1 \end{pmatrix}
        \begin{pmatrix} \pi_x^{m_1-2k_1}& \\ &1 \end{pmatrix}
        \begin{pmatrix}1&\sum_{i=k_1+1}^{+\infty}a_i\pi_x^{i+k_1-m_1} \\ &1 \end{pmatrix} \]
 if  $m_1-2k_1 \leq 1$, and
  \[
M =     \begin{pmatrix}1&\sum_{i=0}^{k_1}a_i\pi_x^{i-k_1} \\ &1 \end{pmatrix}
        \begin{pmatrix} \pi_x^{m_1-2k_1}&\sum_{i=1}^{m_1-2k_1-1}a_{i+k_1}\pi_x^{i} \\ &1 \end{pmatrix} \begin{pmatrix}1&\sum_{i=m_1-2k_1}^{+\infty}a_i\pi_x^{i+k_1-m_1} \\ &1 \end{pmatrix}
\]
if $m_1-2k_1 > 1$. Therefore, 
\[ M \sim 
\begin{cases}
      \begin{pmatrix} \pi_x^{m_1-2k_1}& \\ &1 \end{pmatrix}
     = p_{-(m_1-2k_1)} & \mbox{ if } m_1-2k_1 \leq 1  \\\\
       \begin{pmatrix} \pi_x^{m_1-2k_1}&\sum_{i=1}^{m_1-2k_1-1}a_{i+k_1}\pi_x^{i} \\ &1 \end{pmatrix} & \mbox{ if } m_1-2k_1 > 1. 
 \end{cases}           
\]

In the second situation we are back to the initial problem with $m_2:=m_1-2k_1$ instead of $m_1$ and $a_{i+k_1}$ instead of $b_i$. The idea now is to iterate this process until we get to the first situation, i.e. $m_k \leq 1$. Since $k_1\geq1$, we have $m_2<m_1$. After a finite number $j-1$ of iterations where we define $m_i$ and $k_i$ for $i \geq j$ with $m_i\geq0$ if $i < j$, we will finally get $m_j \leq 1$.  Thus, we are left with the first situation in the previous computations, which yields 
\[\begin{pmatrix} \pi_x^{m_1}&b_1\pi_x+...+b_{d-1}\pi_x^{m_1} \\ &1 \end{pmatrix}\sim p_{-m_jx}.\]
But we know for $n\geq1$ that  $-m_j\leq k_{j-1}\leq m_{j-1}\leq m_1 =d-n\leq d-1$. Then the class of the matrix of our interest is $c_{m_jx}=c_{|m_j|x}\in \mathfrak{N}_d$, if $n\geq1$.

\end{proof}

\begin{lemma}
 The sum of the multiplicities of all edges with origin in a fixed vertex is $q^d+1$, where $d = |y|$.
\end{lemma}

\begin{proof} This follows from \cite[Prop. 2.3]{oliver-graphs}. See \cite[Thm. 2.6]{alvarenga19} for a generalization. 
\end{proof}

Next proposition describes the graph $\cG_{\Phi_y,K}$ in general.  

\begin{prop} \label{prop-edgesnonramified}
Let $y \in |\P^1|$ of degree $d$. The following holds in the graph of $\Phi_y$. 
\begin{enumerate}[(i)]

        \item If $n \geq d$, $c_{nx}$ has exactly two $\Phi_y$-neighbors, which are $c_{(n-d)x}$ with multiplicity $q^d$ and $c_{(n+d)x}$ with multiplicity 1.\medskip 
        
        \item  Every vertex $c_{nx} \in \mathfrak{N}_d -\{c_0\}$ has one, and only one, $\Phi_y$-neighbor outside $\mathfrak{N}_d$. Namely $c_{(n+d)x}$. Moreover the multiplicity of the edge from $c_{nx}$  to $c_{(n+d)x}$ is $1$.\medskip
        
        \item The neighbors of $c_0$ are in $\mathfrak{N}_d$, unless $c_d$, which is $\Phi_y$-neighbor of $c_0$ with multiplicity $q+1$.
   
\end{enumerate}
\end{prop}

\begin{proof} First, we observe that for every $n\geq d$ and every $b_0,...,b_{d-1} \in \FF_q$ 
    \begin{align*}
\begin{pmatrix} \pi_x^{d-n}&b_0+b_1\pi_x+...+b_{d-1}\pi_x^{d-1} \\ &1 \end{pmatrix}&= p_{(n-d)x}\begin{pmatrix} 1&(b_0+b_1\pi_x+...+b_{d-1}\pi_x^{d-1})\pi_x^{n-d} \\ &1 \end{pmatrix}\\
    & \sim p_{(n-d)x}   
\end{align*}
In addition, for every $n \in \N$
\[\begin{pmatrix} \pi_x^{-n}& \\ &\pi_x^d \end{pmatrix}=\begin{pmatrix} \pi_x^{d}& \\ &\pi_x^d \end{pmatrix}\begin{pmatrix} \pi_x^{-(n+d)}& \\&1\end{pmatrix}\sim p_{(n+d)x}.\]
Therefore, the proposition follows from the previous lemma and the above observations. 
\end{proof}


We wish to express the last proposition in terms of the graph $\mathcal{G}_{\Phi_y,K}$. In order to do that, we first establish the following drawing conventions, which holds true for every graph $\mathcal{G}_{\Phi,K'}$.  The vertices $g \in \mathrm{Vert}\; \mathcal{G}_{\Phi,K'}$ are represented by 
labelled dots, and an edge $(g,h,m) \in  \mathrm{Edge}\; \mathcal{G}_{\Phi,K'}$ together with its origin $g$ and its terminus $h$ is drawn as Figure \ref{fig1}.

\begin{center}

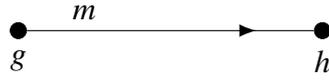
\begin{figure}[h]
\centering

\begin{minipage}[b]{0.45\linewidth}
\[
  \beginpgfgraphicnamed{tikz/fig1}
  \begin{tikzpicture}[>=latex, scale=2]
        \vertex[circle,fill,label={below:$g$}](00) at (0,0) {};
        \vertex[circle,fill,label={below:$h$}](10) at (2,0) {};
    \path[-,font=\scriptsize]
   (00) edge[->-=0.8] node[pos=0.2,auto,black] {\normalsize $m$} (10)
   ;
   \end{tikzpicture}
 \endpgfgraphicnamed
\]
\end{minipage} 
 \caption{The edge  $(g,h,m)$ in $\mathcal{G}_{\Phi,K'}$.}
  \label{fig1}
\end{figure}

\end{center}

Moreover, if the edges  $(g,h,m_1)$ and  $(h,g,m_2)$ in $\mathcal{G}_{\Phi,K'}$ exist, i.e.\ $m_1 m_2 \neq 0$, we draw both of them all at once as in the Figure \ref{fig2}. 
\begin{center}

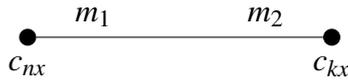
\begin{figure}[h]
\centering

\begin{minipage}[b]{0.45\linewidth}
\[
  \beginpgfgraphicnamed{tikz/fig2}
  \begin{tikzpicture}[>=latex, scale=2]
        \vertex[circle,fill,label={below:$c_{nx}$}](00) at (0,0) {};
        \vertex[circle,fill,label={below:$c_{kx}$}](10) at (2,0) {};
    \path[-,font=\scriptsize]
   (00) edge[-=0.8] node[pos=0.2,auto,black]{\normalsize $m_1$} 
   node[pos=0.8,auto,black]{\normalsize $m_2$}  (10)
   ;
   \end{tikzpicture}
 \endpgfgraphicnamed
\]
\end{minipage} 
 \caption{The edges $(g,h,m_1)$ and  $(h,g,m_2)$ in $\mathcal{G}_{\Phi,K'}$.}
  \label{fig2}
\end{figure}

\end{center}

Hence, the previous proposition can be illustrated as it is in the following Figure \ref{fig3}. The area that is delimited by the nucleus $\mathfrak{N}_{d}$ might have some edges, even if nothing is drawn.


\begin{center}
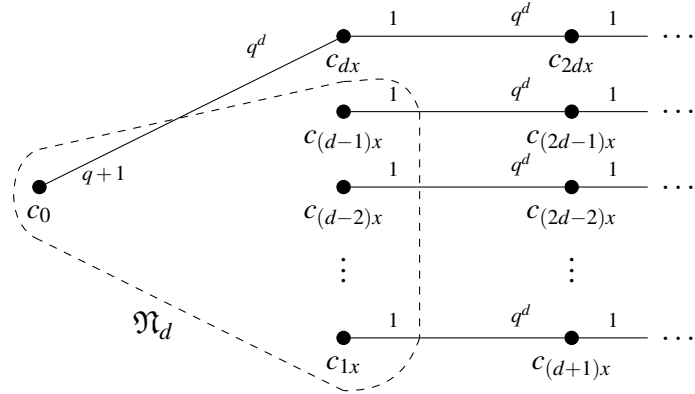
\begin{figure}[h]
\centering

\begin{minipage}[h]{\linewidth}
\[
  \beginpgfgraphicnamed{tikz/fig3}
  \begin{tikzpicture}[>=latex, scale=2]

        \vertex[circle,fill, minimum size = 5pt, label={below:\small $c_{0}$}](c0) at (0,0) {};

      \vertex[circle,fill,minimum size = 5pt, label={below:\small $c_{dx}$}](cd) at (2,1) {};
        \vertex[circle,fill,minimum size = 5pt, label={below:\small $c_{(d-1)x}$}](cd-1) at (2,0.5){};
    \vertex[circle,fill,minimum size = 5pt, label={below:\small $c_{(d-2)x}$}](cd-2) at (2,0) {};
    \vertex[circle,fill,minimum size = 5pt, label={below:\small $c_{1x}$}](c1) at (2,-1) {};
   \vertex[circle,fill,minimum size = 5pt, label={below:\small  $c_{2dx}$}](c2d) at (3.5,1) {};
        \vertex[circle,fill,minimum size = 5pt, label={below:\small $c_{(2d-1)x}$}](c2d-1) at (3.5,0.5){};
    \vertex[circle,fill,minimum size = 5pt, label={below:\small $c_{(2d-2)x}$}](c2d-2) at (3.5,0) {};
    \vertex[circle,fill,minimum size = 5pt, label={below:\small $c_{(d+1)x}$}](c1+d) at (3.5,-1) {};
   

\node at (2,-.5) {$\vdots$};

\node at (3.5,-.5) {$\vdots$};


  \path[-,font=\scriptsize]
   (c0) edge[-=0.8] node[pos=0.2,auto,below]{\tiny $q+1$} 
    node[pos=0.8,auto,black]{\tiny $q^d$}  (cd) ;

   \path[-,font=\scriptsize]
   (cd) edge[-=0.8] node[pos=0.2,auto,black]{\tiny $1$} 
    node[pos=0.8,auto,black]{\tiny $q^d$}  (c2d) ;
      \path[-,font=\scriptsize]
   (cd-1) edge[-=0.8] node[pos=0.2,auto,black]{\tiny $1$} 
    node[pos=0.8,auto,black]{\tiny $q^d$}  (c2d-1) ;
      \path[-,font=\scriptsize]
   (cd-2) edge[-=0.8] node[pos=0.2,auto,black]{\tiny $1$} 
    node[pos=0.8,auto,black]{\tiny $q^d$}  (c2d-2) ;
      \path[-,font=\scriptsize]
   (c1) edge[-=0.8] node[pos=0.2,auto,black]{\tiny $1$} 
    node[pos=0.8,auto,black]{\tiny $q^d$}  (c1+d) ;

   \path[-,font=\scriptsize]
   (c2d) edge[-=0.8] node[pos=0.5,auto,black]{\tiny $1$}  (4,1) ;
      \path[-,font=\scriptsize]
   (c2d-1) edge[-=0.8] node[pos=0.5,auto,black]{\tiny $1$} (4,0.5) ;
      \path[-,font=\scriptsize]
   (c2d-2) edge[-=0.8] node[pos=0.5,auto,black]{\tiny $1$} (4,0) ;
      \path[-,font=\scriptsize]
   (c1+d) edge[-=0.8] node[pos=0.5,auto,black]{\tiny $1$} (4,-1) ;
    
\node at (4.2,1) {$\ldots$};
\node at (4.2,0.5) {$\ldots$};
\node at (4.2,0) {$\ldots$};

\node at (4.2,-1) {$\ldots$};



\node[auto,below] at (0.75,-0.75) { $\mathfrak{N}_d $};

\draw[dashed] (0, 0.25) -- (2,0.7); 
\draw[dashed] (2.5,0.5) -- (2.5, -1);
\draw[dashed]  (2, -1.35) -- (0,-.35);

\draw[dashed] (0,-.35) .. controls (-.25,-.25) and (-.20,.20) .. (0,.25) ;
\draw[dashed] (2,0.7) .. controls (2.2,0.7) and (2.4,0.8) .. (2.5,0.5) ;
\draw[dashed] (2.5,-1) .. controls (2.4,-1.3) and (2.2,-1.35) .. (2,-1.35) ;

   \end{tikzpicture}
 \endpgfgraphicnamed
\]
\end{minipage} 
 \caption{Graph for $\mathcal{G}_{\Phi_y,K}$, with $|y|=d$.}
  \label{fig3}
\end{figure}

\end{center}

We can now state and prove the main theorem of this section, which compute the dimension of $\Phi_y$-eigenspaces.


\begin{thm} \label{thm-mainunramified}
Let $y \in |\P^1|$ of degree $d$ and $\lambda \in \C^{\times}$. 
The map
\[\begin{array}{cccc}
       \cA^{K}(\Phi_y, \lambda) & \longrightarrow  & \bigoplus_{\mathfrak{N}_d }\C \\[3pt]
   f & \longmapsto & (f(g))_{g \in \mathfrak{N}_d } \\[2pt]
\end{array}\] 
is an isomorphism of $\C$-vector spaces. In particular,  $\dim \cA^{K}(\Phi_y, \lambda)=d.$
\end{thm}

\begin{proof} According to Proposition \ref{prop-edgesnonramified}, for each $l \in \{0,...,d-1\}$, the $\Phi_y$-neighbors of $c_{(kd+l)x}$ are all in $\cN_{kd+l}$, unless $c_{((k+1)d+l)x}$.

Let $k\geq 0$ and $f,f' \in \cA^{K}(\Phi_y, \lambda)$ such that $f_{|\cN_{(k+1)d-1}}=f_{|\cN_{(k+1)d-1}}'$. Since $f,f'$ are  $\Phi_y$-eigenforms, we can write 
\[\lambda f(c_{(kd+l)x})=\Phi_y(f)(c_{(kd+l)x})= f(c_{((k+1)d+l)x})-\sum_{g \in \cN_{kd+l}}m(c_{(kd+l)x},g)f(g)\]
and
\[\lambda f'(c_{(kd+l)x})=\Phi_y(f')(c_{(kd+l)x})= f'(c_{((k+1)d+l)x})-\sum_{g \in \cN_{kd+l}}m(c_{(kd+l)x},g)f'(g).\]

The assumption $f_{|\cN_{(k+1)d-1}}=f_{|\cN_{(k+1)d-1}}'$ together above identities imply 
\[ f_{|\cN_{kd+l}}=f_{|\cN_{kd+l}}' \quad \text{ and } \quad f(c_{((k+1)d+l)x})=f'(c_{((k+1)d+l)x}).\]  
Where we use that the graph of $\Phi_y$ (i.e. the multiplicities) does not depend on the automorphic form. 

The above is true for every $l \in \{0,...,d-1\}$, then $f_{|\cN_{(k+2)d-1}}=f_{|\cN_{(k+2)d-1}}'$. Hence, we conclude by induction that $f_{|\mathfrak{N}_{d}}=f_{|\mathfrak{N}_{d}}'$, which implies $f=f'$ and therefore the injectivity.


Let $f$ be a function defined on $\mathfrak{N}_{d}$. We show the surjectivity by showing that we can extend $f$ to an element in $\cA^{K}(\Phi_y, \lambda)$. We define
\[f(c_{dx}):=\frac{1}{q+1}(\lambda f(c_0)-\sum_{g \in \mathfrak{N}_{d}}m(c_{0},g)f(g)).\]
The above definition yields $\Phi_y(f)(c_{0})=\lambda f(c_{0})$. At this point $f$ is defined on $\cN_d$.
Next, for every $l \in \{1,...,d-1\}$, we define 
\[f(c_{(d+l)x}):=\lambda f(c_0)-\sum_{g \in \mathfrak{N}_{d}}m(c_{0},g)f(g),\]
which gives $\Phi_y(f)(c_{kx})=\lambda f(c_{kx})$.
If $n \geq 2d$ we just define by induction 
\[f(c_{nx}):=\lambda f(c_{(n-d)x})-qf(c_{(n-2d)x})\]
The condition of moderate growth follows from the definitions. Therefore, $f \in \cA^K$ and Proposition \ref{prop-edgesnonramified} yields that $f \in \cA^{K}(\Phi_y, \lambda)$, which concludes the proof.
\end{proof}


\begin{cor} Let $y$ be a degree $d$ place of $\P^1$ and $\lambda \in \C^{\times}$. Then $f \in \cA^{K}(\Phi_y, \lambda)$ is completely determined by $\mathfrak{N}_{d}$, the nucleus of $\cG_{\Phi_y,K}$. 
$\hfill \square$
\end{cor}

We finish this section with the drawing graph for a place of degree $d=6$. See Figure \ref{fig4}. The graphs for $d=1,2,3,4,5$ are available in \cite{oliver-graphs}.


\begin{center}
\begin{figure}[h]
\centering
\begin{minipage}[h]{\linewidth}
\[
  \beginpgfgraphicnamed{tikz/fig4}
  \begin{tikzpicture}[>=latex, scale=2]
        \vertex[circle,fill, minimum size = 5pt, label={ [label distance=-0.1cm] left:\small $c_{0}$}](c0) at (0,1.25) {};

        \vertex[circle,fill, minimum size = 5pt, label={ [label distance=-0.1cm] left:  \small $c_{1x}$}](c1) at (.75,-1.25) {};

      \vertex[circle,fill,minimum size = 5pt, label={below:\small $c_{6x}$}](c6) at (2,2) {};
        \vertex[circle,fill,minimum size = 5pt, label={ [label distance=-0.15cm] below right:\small $c_{4x}$}](c4) at (2,1.25){};
    \vertex[circle,fill,minimum size = 5pt, label={below:\small $c_{2x}$}](c2) at (2,0.5) {};
     \vertex[circle,fill,minimum size = 5pt, label={below:\small $c_{7x}$}](c7) at (2,-.5) {};
    \vertex[circle,fill,minimum size = 5pt, label={below:\small $c_{5x}$}](c5) at (2,-1.25) {};
    \vertex[circle,fill,minimum size = 5pt, label={below:\small $c_{3x}$}](c3) at (2,-2) {};
    \vertex[circle,fill,minimum size = 5pt, label={below:\small $c_{12x}$}](c12) at (3.5,2) {};
        \vertex[circle,fill,minimum size = 5pt, label={below:\small $c_{10x}$}](c10) at (3.5,1.25){};
    \vertex[circle,fill,minimum size = 5pt, label={below:\small $c_{8x}$}]
    (c8) at (3.5,0.5) {};
     \vertex[circle,fill,minimum size = 5pt, label={below:\small $c_{13x}$}](c13) at (3.5,-.5) {};
    \vertex[circle,fill,minimum size = 5pt, label={below:\small $c_{11x}$}](c11) at (3.5,-1.25) {};
    \vertex[circle,fill,minimum size = 5pt, label={below:\small $c_{9x}$}](c9) at (3.5,-2) {};


  \path[-,font=\scriptsize]
   (c0) edge[-=0.8] node[pos=0.2, above]{\tiny $q+1$} 
    node[pos=0.8,above,black]{\tiny $q^6$}  (c6) ;
    
      \path[-,font=\scriptsize]
   (c0) edge[-=0.8] node[pos=0.35,above=-.1]{{ \tiny $q^3-q$}} 
    node[pos=0.8,above=-.1,black]{\tiny $q^6-q^5$}  (c4) ;
    
      \path[-,font=\scriptsize]
   (c0) edge[-=0.8] node[pos=0.2,auto,black, below=0.1]{{ \tiny $q^5-q^3$}} 
    node[pos=0.7,below=0.1,black]{\tiny $q^6- q^5$}  (c2) ;

    \draw[black] (-.25,1.25) circle  (0.25cm) node at (-.25,1.6) {\tiny $q^6-q^5$};

    \path[-,font=\scriptsize]
   (c1) edge[-=0.8] node[pos=0.2,above,black]{\tiny $1$} 
    node[pos=0.8,above,black]{\tiny $q^6$}  (c7) ;
  \path[-,font=\scriptsize]
   (c1) edge[-=0.8] node[pos=0.3, above=-.1]{\tiny $q^2$} 
    node[pos=0.8,above=-.1,black]{\tiny $q^6$}  (c5) ;
     \path[-,font=\scriptsize]
      (c1) edge[-=0.8] node[pos=0.2,below=0.15]{\tiny $q^4-q^2$} 
    node[pos=0.6,below=0.15,black]{\tiny $q^6-q^4$}  (c3) ;

        \draw[black] (.5,-1.25) circle  (0.25cm) node at (.5,-.9) {\tiny $q^6-q^4$};

   \path[-,font=\scriptsize]
   (c6) edge[-=0.8] node[pos=0.2,auto,black]{\tiny $1$} 
    node[pos=0.8,auto,black]{\tiny $q^6$}  (c12) ;

      \path[-,font=\scriptsize]
   (c4) edge[-=0.8] node[pos=0.2,auto,black]{\tiny $1$} 
    node[pos=0.8,auto,black]{\tiny $q^6$}  (c10) ;

          \path[-,font=\scriptsize]
   (c2) edge[-=0.8] node[pos=0.2,auto,black]{\tiny $1$} 
    node[pos=0.8,auto,black]{\tiny $q^6$}  (c8) ;

           \path[-,font=\scriptsize]
   (c2) edge[-=0.8] node[pos=0.22,left=-.1,black]{\tiny $q^3$} 
    node[pos=0.75,left=-.1,black]{\tiny $q^5$}  (c4) ;

    \draw[black] (2,0.25) circle  (0.25cm) node at (2,-.09) {\scalebox{0.65}{ $q^5-q^3$}};

      \path[-,font=\scriptsize]
   (c5) edge[-=0.8] node[pos=0.2,auto,black]{\tiny $1$} 
    node[pos=0.8,auto,black]{\tiny $q^6$}  (c11) ;
    
          \path[-,font=\scriptsize]
   (c3) edge[-=0.8] node[pos=0.2,auto,black]{\tiny $1$} 
    node[pos=0.8,auto,black]{\tiny $q^6$}  (c9) ;
    
     \path[-,font=\scriptsize]
   (c7) edge[-=0.8] node[pos=0.2,auto,black]{\tiny $1$} 
    node[pos=0.8,auto,black]{\tiny $q^6$}  (c13) ;

  \draw[black] (2,-2.25) circle  (0.25cm) node at (2,-2.6) {\tiny $q^4$};

   \path[-,font=\scriptsize]
   (c12) edge[-=0.8] node[pos=0.5,auto,black]{\tiny $1$}  (4,2) ;
   
      \path[-,font=\scriptsize]
   (c10) edge[-=0.8] node[pos=0.5,auto,black]{\tiny $1$} (4,1.25) ;
   
      \path[-,font=\scriptsize]
   (c8) edge[-=0.8] node[pos=0.5,auto,black]{\tiny $1$} (4,0.5) ;
   
      \path[-,font=\scriptsize]
   (c13) edge[-=0.8] node[pos=0.5,auto,black]{\tiny $1$} (4,-.5) ;
   
       \path[-,font=\scriptsize]
   (c11) edge[-=0.8] node[pos=0.5,auto,black]{\tiny $1$} (4,-1.25) ;
   
       \path[-,font=\scriptsize]
   (c9) edge[-=0.8] node[pos=0.5,auto,black]{\tiny $1$} (4,-2) ;
    
\node at (4.2,2) {$\ldots$};
\node at (4.2,1.25) {$\ldots$};
\node at (4.2,.5) {$\ldots$};

\node at (4.2,-.5) {$\ldots$};
\node at (4.2,-1.25) {$\ldots$};
\node at (4.2,-2) {$\ldots$};


\node[auto,below] at (0.75,-2.5) { $\mathfrak{N}_6$};

\draw[dashed] (-.7, 1.7) -- (-.7,-2.8); 
\draw[dashed] (2.6,1.7) -- (2.6, -2.8);
\draw[dashed]  (-.7, 1.7) -- (2.6,1.7);
\draw[dashed]  (-.7, -2.8) -- (2.6,-2.8);

   \end{tikzpicture}
 \endpgfgraphicnamed
\]
\end{minipage} 
 \caption{The graph $\cG_{\Phi_y,K}$ for $|y|=6$.}
  \label{fig4}
\end{figure}
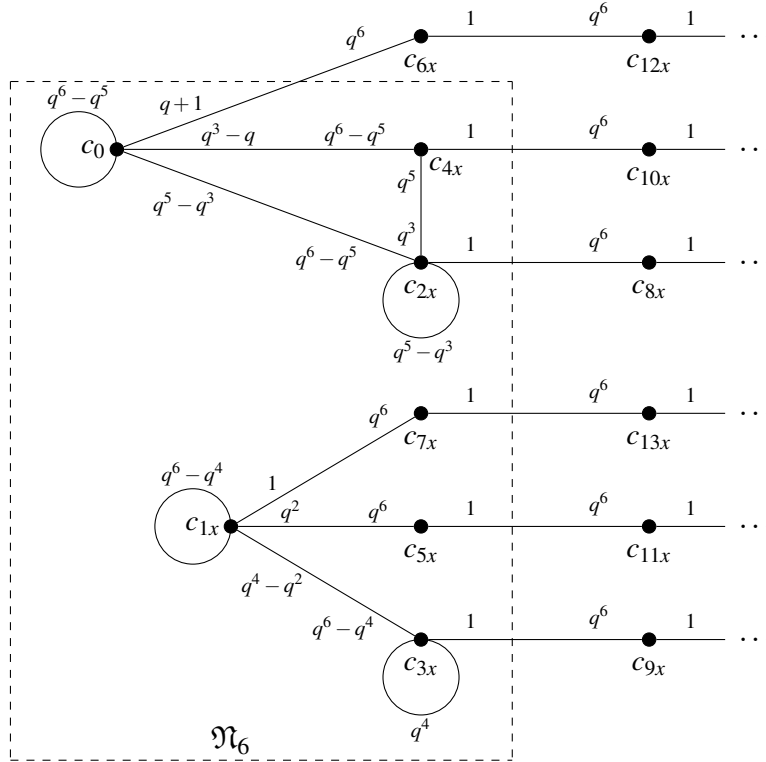

\end{center}


\section{Ramified Hecke eigenspaces}
\label{sec-ramified}

As before, we fix $x \in |\P^1|$ a place of degree~$1$.
Let us consider, through this section,
\[ K' := K_x'\times \prod_{y\neq x}K_y \quad \text{ such that } \quad K_x':=\big\{ k \in K_x \;|\; k \equiv \id~\textrm{mod}(\pi_x)\big\}.\]
The goal is to generalize the previous section to the ramified case attached to $K' \in \cV$. In the next section we investigate the cases where the degree of $x$ is bigger than $1$. 


\subsection*{Representation of the double coset}
The first step to describe the action of operators on $\cH_{K'}$ over $\cA^{K'}$ is to find a system of representatives of $G(F) \setminus G(\mathbb{A})/Z K'$. In order to do that, we consider the canonical projection 
\[ P : G(F) \setminus G(\mathbb{A})/Z K' \longrightarrow G(F) \setminus G(\mathbb{A})/Z K .\]
and observe that it is sufficient to find a system of representatives for every fiber $P^{-1}(c_{nx}),~n \in \N.$
For this purpose, we generalize the Strong Approximation for $\mathrm{SL}_2$ given on previous section in the following proposition. 

\begin{rem} Even if there are now two cosets, we will keep denoting just by $g$ for the class of an element $ g\in G(\A)$ in both of them, with the context being clear every time. Moreover, we keep denoting $g \sim g'$ whether
$g, g' \in G(\A)$ are equivalent in $G(F) \setminus G(\mathbb{A})/Z K'.$
\end{rem}

\begin{prop} \label{prop-ramifstrongaprox} Let $x \in |\P^1|$ of degree $1$ and $\Gamma_x=\bigcap_{y\neq x} (G(F) \cap K_y ) =\GL_2(\FF_q[\pi_x^{-1}]) $.
Then the map
\[ (.)_x : \Gamma_x \setminus G_x/Z_xK_x' \longrightarrow G(F) \setminus G(\mathbb{A})/Z K'\]
which is defined by completing the adele with the identity matrix in the places different
from $x$, is well-defined and injective.
\end{prop}

\begin{proof}
    We give here an elementary proof that works exactly the same way for $K=K'$.
    First we prove that 
    \[  G(F) \setminus G(\mathbb{A})/Z K'= G(F) \setminus G(\mathbb{A})/Z_{x}K'.\]
    It is sufficient to show that if $h=z_xgh'k'\in Z_xG(F)h'K'$, we still have $zh\in Z_xG(F)h'K'$, for every $z \in Z$. By changing $z_x$ we can suppose $z$ is trivial at $x$. By changing the $k'$, we can suppose that $z$ is trivial everywhere except on a finite number of places, where it is  $1/\rho(\pi)$, for $\pi$ the uniformizer of the corresponding place and $\rho$ a polynomial on it. Thus, by multiplying $g$  by the finite product of the value of $z$ on places where it is non-trivial, we can finally suppose that $z$ is trivial everywhere, which is the desired.

    Next, we prove the desired identity 
    \[ G(F) \setminus G(\mathbb{A})/Z_{x}K'=\Gamma_x \setminus G_x/Z_xK_x'.\]
    Suppose that we have $a=z\gamma bk_x'$ with $a,b \in G_x$, $\gamma \in \Gamma_x$, $k_x' \in K_x'$. We define $k' \in K'$ such that $k'$ is equal to $\gamma^{-1}$ on every place $y\neq x$, and equal to $k_x'$ at $x$. Then we have $(a)_x=(z)_x\gamma (b)_xk'$ with $\gamma \in G(F)$ and $k' \in K'$. Reciprocally, suppose now that we have $(a)_x=(z)_xg(b)_xk'$ with $a,b \in G_x$, $g \in G(F)$ and $k' \in K'$. This relation, projected on the place $x$, gives $a=zgbk_x'$, and projected on every other place $y$, gives $gk_y'=\id$. Then we deduce $g \in \bigcap_{y\neq x}K_y$, which shows that $g \in \Gamma_x$ and concludes the proof.
\end{proof}

\begin{lemma}\label{lemma3.3} Let $P : G(F) \setminus G(\mathbb{A})/Z K' \longrightarrow G(F) \setminus G(\mathbb{A})/Z K $ be the canonical projection. Then 
$P^{-1}(c_{nx})=\big\{(p_{nx}g)_x \;|\; g\in \GL_2(\FF_q)\big\}$.
\end{lemma}

\begin{proof}
We consider $\rho_x : K_x \rightarrow \GL_2(\FF_q)$ the morphism of reduction mod($\pi_x$), pre-composed with the projection from $K$ to $K_x$. We observe that $K_x'$ is the kernel of $\rho_x$, which justify that $K_x'$ is of finite index in $K_x$, and give the following decomposition 
\[K_x=\coprod_{g \in \GL_2(\FF_q)}g K_x'.\]
Therefore $P^{-1}(c_{nx})=\big\{(p_{nx}g)_x \;|\; g\in \GL_2(\FF_q)\big\}$.
\end{proof}

\begin{prop} \label{prop-ramifequiv}
Let $n \in \N^{\times}$ and $g,g'\in \GL_2(\FF_q)$. Then
\[p_{nx}g\sim p_{nx}g' \text{ in } \Gamma_x \setminus G_x/Z_xK_x'
\; \text{ if and only if } \; p_{nx}g\sim p_{nx}g' \text{ in } \Gamma_x \setminus G_x.\]
\end{prop}

\begin{proof} We proceed the proof in two steps:
\begin{enumerate}[(i)]
    \item $p_{nx}g\sim p_{nx}g' \text{ in } \Gamma_x \setminus G_x/K_x'$ if, and only if, $p_{nx}g\sim p_{nx}g' \text{ in }\Gamma_x \setminus G_x.$ \medskip
    
    \item $p_{nx}g\sim p_{nx}g' \text{ in } Z_x\Gamma_x \setminus G_x$ if, and only if,  $p_{nx}g\sim p_{nx}g' \text{ in }\Gamma_x \setminus G_x.$
\end{enumerate}
For $\textrm{(i)}$, let $k' \in K_x'$ such that $ p_{nx}gk'(g')^{-1}p_{nx}^{-1} \in \Gamma_x.$
By using the embedding 
\[\GL_2(\FF_q[[\pi_x]]) \hookrightarrow (M_2(\FF_q))[[\pi_x]],\]
we can write 
$k'=\id+\pi_xk'' \text{ with } k'' \in (M_2(\FF_q))[[\pi_x]]$,
where $M_2$ is the group of $2 \times 2$ matrices. Thus, \[p_{nx}g(g')^{-1}p_{nx}^{-1}+\pi_xp_{nx}gk''(g')^{-1}p_{nx}^{-1} \in \Gamma_x.\] 
We denote $a \in \Fq$ and $\ b(\pi_x), c(\pi_x) \in \Fq[[\pi_x]]$  the coefficients in the position $(2,1)$ of the matrices $g(g')^{-1}$, $gk'(g')^{-1}$ and $gk''(g')^{-1}$, respectively. Since $\Gamma_x=\GL_2(\FF_q[\pi_x^{-1}])$ embeds on $(M_2(\FF_q))[\pi_x^{-1}]$,
$b(\pi_x)=0$. But as above, $b(\pi_x)= a +\pi_x c(\pi_x)$, then $a=0$ and  $p_{nx}g(g')^{-1}p_{nx}^{-1} \in \Gamma_x$.

Next we prove $\textrm{(ii)}$. If $p_{nx}g=z\gamma p_{nx}g'$, with $g,g'\in \GL_2(\FF_q)$, $\gamma \in \Gamma_x$ and $z \in F_x$, then by taking the determinant we see that $z^2 \in \FF_q^\times$. Thus $z\gamma \in \Gamma_x$, which concludes the proof.
\end{proof}

Next we can find a necessary and sufficient condition for matrices of the form $p_ng$, with $g \in \GL_2(\FF_q)$, be equivalents in  $G(F) \setminus G(\mathbb{A})/Z K'$. 
Let 
\[ g:=\begin{pmatrix} a&b \\ c&d \end{pmatrix} \in \GL_2(\FF_q), \;
g':=\begin{pmatrix} a'&b' \\ c'&d' \end{pmatrix} \in \GL_2(\FF_q)\] 
and define
\[ h := p_{nx}g(g')^{-1}p_{nx}^{-1}=(a'd'-b'c')^{-1}\begin{pmatrix} ad'-bc'&(ab'-a'b)\pi_x^{-n} \\ (cd'-c'd)\pi_x^{n}&-cb'+da' \end{pmatrix}.\]
If $n \in \N^{\times}$, $h \in \Gamma_x$ if, and only if, $c'd=cd'$. Hence, thanks to Propositions \ref{prop-ramifstrongaprox} and \ref{prop-ramifequiv}, we have a system of representatives for $G(F) \setminus G(\mathbb{A})/ZK'$. Indeed, if we define
\[ \vartheta_{[0:1]}~:=~\begin{pmatrix} &1 \\ 1& \end{pmatrix} \; \text{ and }
\; \vartheta_{[1:a]}:=\begin{pmatrix} 1& \\ a&1 \end{pmatrix} \forall a \in \FF_q, \]
then for $n \in \N^{\times}$, $\big\{ (p_{nx}\vartheta_w)_x \;|\; w \in \P^1(\FF_q)\}$ is a system of representatives of $P^{-1}(c_{nx})$.
Since $p_0=\id \in K'$, $P^{-1}(c_{0})=\big\{ (p_0)_x \big\}.$
Summarizing the above discussion, we have the following theorem.

\begin{thm} \label{thm-ramifiedvertices} The set
    $\big\{\id\big\}\cup \big\{(p_{nx}\vartheta_w)_x \;|\; n\in \N^{\times},~w\in\P^1(\FF_q)\big\}$ is a system of representatives for $G(F) \setminus G(\mathbb{A})/Z K'$.
\end{thm}

\begin{df}
    Let $c_0'$  stand for the classe of $(p_0)_x$ in $G(F) \setminus G(\mathbb{A})/Z K'$. Moreover, for every $n \in \N^{\times}$ and every $w \in \P^1(\FF_q)$, $ c_{nx,w}'$ stands for the class of $(p_{nx}\vartheta_w)_x$ in $G(F) \setminus G(\mathbb{A})/Z K'$.
\end{df}


\subsection*{Graphs of ramified Hecke operators.}

For every $\Phi \in \cH_{K'}$, accordingly Definition \ref{defgraphs}, the set of the vertices of its graph $\cG_{\Phi,K'} $ is $G(F) \setminus G(\mathbb{A})/ZK'$. Hence, by Theorem \ref{thm-ramifiedvertices}
\[ \mathrm{Vert}\; \cG_{\Phi,K'} = \big\{c_0'\big\}\cup \big\{c_{nx,w}' \;|\; n\in \N^{\times} \text{ and } w\in\P^1(\FF_q)\big\}. \]

\begin{df} For every place $y$ of $\P^1$, we define $\Phi_y' \in \cH_{K'}$ be the characteristic function of
\[ K' \begin{pmatrix} \pi_y& \\ &1 \end{pmatrix} K'\]
divided by the volume of $K'$. 
\end{df}

We recall that we see { \tiny $\begin{pmatrix} \pi_y& \\ &1 \end{pmatrix}$ } as an adelic matrix in $G(\A)$ by completing the adele with the identity matrix at all places different of $y$.

\begin{rem} We observe that there are others ramified Hecke operators. However, based on Satake theorem for the unramified Hecke algebra, the above ramified Hecke operator seems to be a natural choice to be investigated.
\end{rem}

To investigate the graphs of the $\Phi_y'$, we need a better understanding of $K' { \tiny \begin{pmatrix} \pi_y& \\ &1 \end{pmatrix} } K'$, which is done in the next proposition. Before, we need to introduce the following notation. Let $y$ be a place of $\P^1$, we define 
 $\xi_{y,w}$ for $w \in \P^1(\kappa(y))$ as follows: 
\[ \xi_{y,[0:1]}:=\begin{pmatrix} 1& \\ &\pi_y \end{pmatrix}\; \text{ and } \; \xi_{y,[1:a]}:=\begin{pmatrix} \pi_y&a \\ &1 \end{pmatrix} \]
for all $a \in \kappa(y)$. 

\begin{prop} \label{propdecomp}Considering above notation, then
    \begin{enumerate}[(i)]
        \item $K' p_y^{-1}  K'=\coprod_{w\in \P^1(\kappa(y))}\xi_{y,w} K'$, for every place $y\neq x$. \medskip
        
        \item $K'p_x^{-1}  K'=\coprod_{a \in \FF_q}\begin{pmatrix} \pi_x&a\pi_x \\ &1 \end{pmatrix} K'.$
    \end{enumerate}
\end{prop}

\begin{proof}
The statements are equivalent to
\[ (i)\; K_yp_y^{-1}K_y=\coprod_{w\in \P^1(\kappa(y))}\xi_{y,w}K_y \quad \text{ and } \quad (ii)\; K_x'p_x^{-1}K_x'=\coprod_{a \in \FF_q}\begin{pmatrix} \pi_x&a\pi_x \\ &1 \end{pmatrix}K_x', \]
respectively. 
We will only prove $(ii)$ here. The proof of $(i)$ is similar and can be proved as the unramified, see \cite[Lemma 3.7]{gelbart-75}.

It follows as in the proof of \cite[Lemma 2.5]{alvarenga19} that we can find a finite number of $\xi \in G_x$ such that $K_x'p_xK_x'$ is the disjoint union of the $\xi K_x'$. Take such a $\xi \in G_x$, it exists $k_1',k_2',k_3' \in K_x'$ such that $k_1'p_xk_2'=\xi k_3'.$ We deduce that $\xi \in M_2(\FF_q[\pi_x])$ and that its reduction mod($\pi_x$) is ${ \tiny \begin{pmatrix} 0& \\ &1 \end{pmatrix}}$. Moreover, By Iwasawa decomposition and the decomposition of $K_x$ with $K_x'$ we know that we can write $\xi=hg$ with $h$ an upper triangular matrix and $g \in \GL_2(\FF_q)$. Since $h$ is reversible, the condition mod($\pi_x$) implies that $g$ is also upper triangular. Then, $\xi$ is upper triangular, we denote $\xi={ \tiny \begin{pmatrix} P&Q \\ &S \end{pmatrix}}$ with $P,Q,S \in \FF_q[\pi_x].$ By taking the determinant, we get $PS=\pi_x$. Then the condition mod($\pi_x$) shows that $P=\pi_x,S=1$. Finally, we note that \[ \begin{pmatrix} \pi_x&Q \\ &1 \end{pmatrix}K'=\begin{pmatrix} \pi_x&R \\ &1 \end{pmatrix}K'\] if and only if $\pi_x$ divides $Q-R$. Then we conclude that 
\[K_x'p_xK_x'\subset \coprod_{a \in \FF_q}\begin{pmatrix} \pi_x&a\pi_x \\ &1 \end{pmatrix}K_x'.\]
For the other inclusion, we observe that 
\[\begin{pmatrix} 1&a\pi_x \\ &1 \end{pmatrix}p_x=\begin{pmatrix} \pi_x&a\pi_x \\ &1 \end{pmatrix}\]which concludes the proof.
\end{proof}

The above proposition allows us to compute the $\Phi_y$-neighbors of $g\in G(F) \setminus G(\mathbb{A})/ZK'$, for every place $y$ of $\P^1.$

\begin{thm} \label{thm-ramifiededges}
Let $x \in |\P^1|$ be the fixed place of degree $1$ where we are considering the ramification $K'$. Let $y$ be a place of  $\P^1$ of degree $d$, $y \neq x$. Then,
    \begin{enumerate}[(i)]
        \item The  $\Phi_y'$-neighbors of $g$ in the graph $\cG_{\Phi_y', K'}$ are the classes of $ g \xi_{y,w}$, with $ w\in \P^1(\kappa(y))$. Moreover, for every $w \in \P^1(\kappa(y))$
        \[m(g,g\xi_{y,w})=\#\big\{w'\in \P^1(\kappa(y))\;|\; g \xi_{y,w'} \sim g\xi_{y,w}\big\}.\]
        In particular, every vertex in $\cG_{\Phi_y', K'}$  has $q^{d}+1$ neighbors when counted with multiplicity. \medskip
        
        \item The  $\Phi_x'$-neighbors of $g$ in the graph $\cG_{\Phi_x', K'}$ are the classes of
        \[ g\begin{pmatrix} \pi_x&a\pi_x \\ &1 \end{pmatrix}, \]
        with $ a\in \FF_q.$ Moreover, for every $a \in \FF_q$ 
    \[ m \big(g, g { \tiny \begin{pmatrix} \pi_x&a\pi_x \\ &1 \end{pmatrix}} \big)= \# 
    \left\lbrace a'\in \FF_q \;\big|\; g \begin{pmatrix} \pi_x&a'\pi_x \\ &1 \end{pmatrix} \sim g\begin{pmatrix} \pi_x&a\pi_x \\ &1 \end{pmatrix} \right\rbrace. \]
           In particular, every vertex in $\cG_{\Phi_x', K'}$ has $q$ neighbors when counted with multiplicity.
    \end{enumerate}
\end{thm} 

\begin{proof}  Let $z$ be any place of $\P^1$, i.e.\ either $z=x$ or $z=y$. Suppose that 
\[ K' \begin{pmatrix}
    \pi_z & \\
     & 1 
\end{pmatrix} K'= \coprod_{i=1}^n g_i K'.\] 
Then, for every $ f \in \cA^{K'}$, 
    \begin{align*}
        \Phi_z(f)(g)&=(1/\textrm{vol }K')\int_{G(\A)} \mathrm{char} (K'\begin{pmatrix}
    \pi_z & \\
     & 1 
\end{pmatrix} K')(h)f(gh)dh\\
        &=(1/\textrm{vol }K')\sum_{i=1}^n \int_{g_i K'} f(gh)dh\\
        &=\sum_{i=1}^nf(g\alpha_i).
    \end{align*}
Thus, the theorem follows from the previous proposition.
\end{proof}

\begin{rem} \label{rem-calculation}
    Here it is an important remark where we explain a method to compute the class of an element  $g_{x} \in G_x$ in $\Gamma_x \setminus G_x/Z_xK_x'$ even if it is not of the form $p_{nx}g$, with $g \in \GL_2(\FF_q)$. The first step is to use the algorithm given by Lorscheid in \cite[App. A]{oliver-graphs} to determine its class in $\Gamma_x \setminus G_x/Z_xK_x$. At this point, we will know a $n \in \N$, a $\gamma \in \Gamma_x$, a $z \in Z_x$ and a $k \in K_x$ such that $g_{x}=z\gamma p_{nx} k$. Then we write $k=gk'$ with $g \in \GL_2(\FF_q)$ and $k' \in K'$. Thus, we have $g_{x}=z\gamma p_{nx}gk'$ and then we know that $g_{x}\sim p_ng$ in $\Gamma_x \setminus G_x/Z_xK_x'$, which allow us to conclude. We see  the convenience of taking fibers by $P$ to deduce a system of representatives from another one, where we already know an algorithm to get every classes. This technique can obviously be generalized for others $K'$, cf. section \ref{sec-generalramification}. We use this technique to draw the graphs.
\end{rem}


\subsection*{Digression on another system of representatives}
Before drawing the graphs, we will show an example of the technique of the previous remark to answer an interrogation the reader might have. We consider $\{c_{nx} \;|\; n \in \N\}$ as a system of representatives of $G(F) \setminus G(\mathbb{A})/Z K$ but, according to \cite[App. A]{oliver-graphs}, we can also consider $\{c_{-nx}\;|\; n \in \N\}$ as a system of representatives. The questions are: why are we making this choice? And what would it be this new system of representative of the ramified double coset? If we want to get a system of representatives of the unramified double coset from this system of representatives, we can do the same thing, and it will appear that a system of representatives of the fiber $P^{-1}(c_{nx})$ will be 
$ \{ p_{-nx}\vartheta_w \;|\; w \in \P^1(\FF_q) \}$ 
with 
\[ \vartheta_{[1:a]}:=\begin{pmatrix} 1&a \\ &1 \end{pmatrix}  \; \text{ and } \; \vartheta_{[0:1]}~:=~\begin{pmatrix} &1 \\ 1& \end{pmatrix},\] 
for $a \in \FF_q.$
Since we now get two systems of representatives for the fiber $P^{-1}(c_n)$ we can ask the bijection between them. To answer this, let us use the technique of the previous remark. We know that $p_{nx}$ and $p_{-nx}$ are equivalent in $G(F) \setminus G(\mathbb{A})/Z K$ because they are equivalent in $\Gamma_x \setminus G_x/Z_xK_x$, namely
$$ p_{nx}=\begin{pmatrix} &1 \\ 1& \end{pmatrix}p_{-nx}\begin{pmatrix} \pi_x^{n}& \\ &\pi_x^{n} \end{pmatrix}\begin{pmatrix} &1 \\ 1& \end{pmatrix}.$$
Hence, multiplying by $\vartheta_w$ on the left we get that $p_{nx}\vartheta_w$ is equivalent to 
\[ p_{-nx}\begin{pmatrix} &1 \\ 1& \end{pmatrix}\vartheta_w\] 
in $\Gamma_x \setminus G_x/Z_xK_x'$. Moreover, observe that 
\[ p_{-nx}\begin{pmatrix} &1 \\ 1& \end{pmatrix}\vartheta_{[1:a]} 
\sim p_{-nx}\vartheta_{[1:a^{-1}]}'\]
for all $a \in \FF_q^{\times}$, 
\[ p_{-nx}\begin{pmatrix} &1 \\ 1& \end{pmatrix}\vartheta_{[1:0]}
\sim p_{-nx}\vartheta_{[1:0]}' \quad \text{ and } \quad 
p_{-nx}\begin{pmatrix} &1 \\ 1& \end{pmatrix}\vartheta_{[0:1]}
\sim p_{-nx}\vartheta_{[0:1]}'.\]
Therefore, yields that $\{p_{-nx}\vartheta_w' \;|\; w \in \P^1(\FF_q) \}$ is a system of representatives of the fiber $P^{-1}(c_n)$ and the exact bijection between the two system or representatives is
\begin{align*} 
    p_{nx}\vartheta_{[1:a]} &\sim p_{-nx}\vartheta_{[1:a^{-1}]}' \\
    p_{nx}\vartheta_{[1:0]} &\sim p_{-nx}\vartheta_{[0:1]}'\\
    p_{nx}\vartheta_{[0:1]} &\sim p_{-nx}\vartheta_{[1:0]}'
 \end{align*}
for all $a \in \FF_q^{\times}.$
We will keep using the system of representatives given by $p_{nx},\; n \in \N$, although we could change the system of representative at any moment to get the most adapted one, thanks to the aforementioned discussion.\\

Next, we describe the graphs of $\Phi_z'$ for every place $z$ of $\P^1$. Consequently, we can calculate the dimension of its eigenspaces. 


\subsection*{Graph and eigenspace of $\Phi_x'$}
We recall that  $x$ is the fixed place of $\P^1$, where we consider the ramification in $K'$.
Let $n \in \N^{\times}$, $w \in \P^1(\FF_q)$ and $a \in \FF_q$. From Theorems \ref{thm-ramifiedvertices} and \ref{thm-ramifiededges}, in order to describe the edges in $\cG_{\Phi_x', K'}$ we have to find the class of 
\[ p_{nx}\vartheta_w\begin{pmatrix} \pi_x&a\pi_x \\ &1 \end{pmatrix} \]
in $G(F) \setminus G(\mathbb{A})/ZK'$, which is equivalent by Proposition \ref{prop-ramifstrongaprox} to find the class in $\Gamma_x \setminus G_x/Z_xK_x'$. Using the strategy in Remark \ref{rem-calculation},

\begin{align*}
     p_{nx}\vartheta_{[0:1]}\begin{pmatrix} \pi_x&a\pi_x \\ &1 \end{pmatrix}&=\begin{pmatrix} \pi_x^{-n}& \\ &\pi_x \end{pmatrix}\begin{pmatrix} &1 \\ 1&a \end{pmatrix}\sim 
        p_{(n+1)x}\vartheta_{[1:a^{-1}]} \\
     &\\
  p_{nx}\vartheta_{[1:b]}\begin{pmatrix} \pi_x&a\pi_x \\ &1 \end{pmatrix}&
  =p_{(n-1)x}
   \begin{pmatrix} 1&a \\ b\pi_x&ab\pi_x +1 \end{pmatrix}\sim  p_{(n-1)x} \\
         &\\
  p_{0}\begin{pmatrix} \pi_x&a\pi_x \\ &1 \end{pmatrix}&=p_{x}\begin{pmatrix} 1&a \\ &1 \end{pmatrix}\sim 
        p_{x}.\ \\
\end{align*}
Which give us the full graph $\cG_{\Phi_x', K'}$, see Figure \ref{fig5}. We use the explicitly description of  $\cG_{\Phi_x', K'}$ in the following to investigate the eigenspace $\cA^{K'}(\Phi_x', \lambda)$, for  $\lambda \in \C^{\times}$. 


\begin{center}
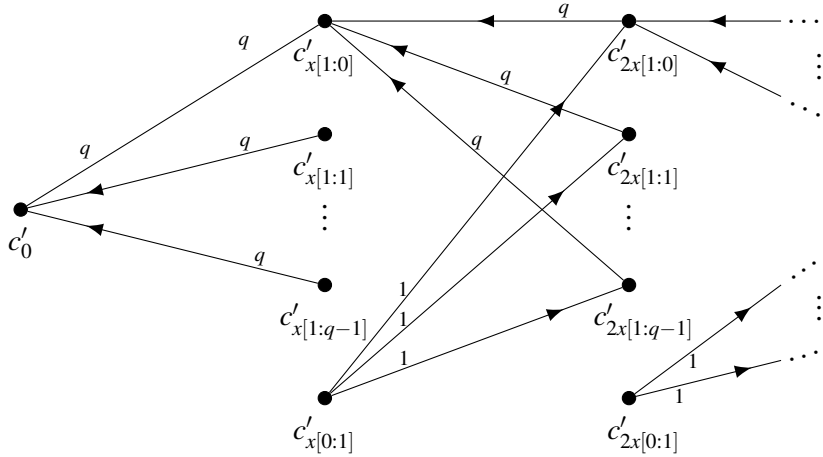
\begin{figure}[h]
\centering
\begin{minipage}[h]{\linewidth}
\[
  \beginpgfgraphicnamed{tikz/fig5}
  \begin{tikzpicture}[>=latex, scale=2]

        \vertex[circle,fill, minimum size = 5pt, label={below:\small $c_{0}'$}](c0) at (0,0) {};

      \vertex[circle,fill,minimum size = 5pt, label={below:\small 
      $c_{x[1:0]}'$}](cx10) at (2,1.25) {};
        \vertex[circle,fill,minimum size = 5pt, label={below:\small 
        $c_{x[1:1]}'$}](cx11) at (2,0.5){};
    \vertex[circle,fill,minimum size = 5pt, label={below:\small 
    $c_{x[1:q-1]}'$}](cx1q-1) at (2,-0.5) {};
    \vertex[circle,fill,minimum size = 5pt, label={below:\small 
    $c_{x[0:1]}'$}](cx01) at (2,-1.25) {};

      \vertex[circle,fill,minimum size = 5pt, label={below,right=0.2:\small 
      $c_{2x[1:0]}'$}](c2x10) at (4,1.25) {};
        \vertex[circle,fill,minimum size = 5pt, label={below,right=0.2:\small 
        $c_{2x[1:1]}'$}](c2x11) at (4,0.5){};
    \vertex[circle,fill,minimum size = 5pt, label={below,right=0.2:\small 
    $c_{2x[1:q-1]}'$}](c2x1q-1) at (4,-0.5) {};
    \vertex[circle,fill,minimum size = 5pt, label={below,right=0.2:\small 
    $c_{2x[0:1]}'$}](c2x01) at (4,-1.25) {};
   

\node at (2,0) {$\vdots$};

\node at (4,0) {$\vdots$};


  \path[-,font=\scriptsize]
   (c0) edge[-=0.8] node[pos=0.2,auto,above]{\tiny $q$} 
    node[pos=0.8,auto,black]{\tiny $q$}  (cx10) ;
 \path[-,font=\scriptsize] (cx11) edge[->-=0.8] node[pos=.25,above=-0.1,black]{\tiny $q$} node[pos=0.8,auto,black]{}  (c0)  ;
  \path[-,font=\scriptsize] (cx1q-1)edge[->-=0.8] node[pos=0.2,above=-0.1,black]{\tiny $q$} 
  node[pos=0.8,auto,black]{}  (c0)   ;


   \path[-,font=\scriptsize] (cx01) edge[->-=0.8] node[pos=0.25,above=-0.08,black]{\tiny $1$}  (c2x10) ;
    \path[-,font=\scriptsize] (cx01) edge[->-=0.8] node[pos=0.25,above=-0.1,black]{\tiny $1$} (c2x11) ;
 \path[-,font=\scriptsize] (cx01) edge[->-=0.8] node[pos=0.25,above=-0.1,black]{\tiny $1$} (c2x1q-1) ;

 \path[-,font=\scriptsize]  (c2x10) edge[->-=0.5] node[pos=0.2,above=-0.1,black]{\tiny $q$}  (cx10) ;
 \path[-,font=\scriptsize]  (5,1.25) edge[->-=0.5] node[pos=0.2,auto,black]{\tiny }  (c2x10) ;
  \path[-,font=\scriptsize]  (5,0.75) edge[->-=0.5] node[pos=0.2,auto,black]{\tiny }   (c2x10);
   \path[-,font=\scriptsize] (c2x11) edge[->-=0.8] node[pos=0.4,above=-0.1,black]{\tiny $q$}  (cx10) ;
    \path[-,font=\scriptsize] (c2x1q-1) edge[->-=0.8] node[pos=0.5,above=-0.1,black]{\tiny $q$} (cx10) ;
 \path[-,font=\scriptsize] (c2x01) edge[->-=0.8] node[pos=0.4,below=-0.1,black]{\tiny $1$} (5,-.5) ;
 \path[-,font=\scriptsize] (c2x01) edge[->-=0.8] node[pos=0.3,below=-0.1,black]{\tiny $1$} (5,-1) ;


\node at (5.15,1.25) { $\ldots$};
\node[rotate=11] at (5.15,0.72) { $\ddots$};
\node[rotate=15] at (5.15,-.97){ $\ldots$};
\node[rotate=35] at (5.15,-.4){ $\ldots$};
\node at (5.25,1) { $\vdots$};
\node at (5.25,-.6) { $\vdots$};

   \end{tikzpicture}
 \endpgfgraphicnamed
\]
\end{minipage} 
 \caption{The graph $\mathcal{G}_{\Phi_{x}',K'}$.}
  \label{fig5}
\end{figure}

\end{center}


\begin{thm} \label{thm-mainramified1}
    If $\lambda \in \C^{\times}$, the map 
    \[ \begin{array}{cccc}
   \cA^{K'}(\Phi_x', \lambda) & \longrightarrow & \C \\[1pt]
   f & \longmapsto & f(c_0') \\
\end{array} \] 
is an isomorphism of $\C$-vector spaces. In particular, $\dim \cA^{K'}(\Phi_x', \lambda)=1$.
\end{thm}

\begin{proof}
    Let $n\geq 2$ be an integer. For every $f \in \cA^{K'}(\Phi_x', \lambda)$ we have the following relations from the graph $\cG_{\Phi_x',K'}$
    \begin{align*}
   f(c_{nx,[1:a]}') & = \frac{q}{\lambda}f(c_{(n-1)x,[1:0]}')= \frac{q^n}{\lambda^n}f(c_0'), \\
   f(c_{nx,[0:1]}') & = \frac{1}{\lambda}\sum_{a\in \FF_q}f(c_{(n+1)x,[1:a]}')=\frac{q^2}
   {\lambda^2}f(c_{nx,[1,0]}')=\frac{q^3}{\lambda^3}f(c_{(n-1)x',[1,0]}')=\frac{q^{n+2}}{\lambda^{n+2}}f(c_0')
\end{align*}
for all $a \in \FF_q.$ 
Which proves that if $ f,f'\in \cA^{K'}(\Phi_x', \lambda)$, then $f=f'$ if, and only if, $f(c_0')=f'(c_0')$. 
To conclude, consider the unique function 
\[f : G(F) \setminus G(\mathbb{A})/Z K'\longrightarrow \C\] 
such that $f(c_0')=1$ and the other values of $f$ are the ones predicated by the above relations.  By construction, if $f \in \cA^{K'}$, then it will automatically be in the eigenspace  $\cA^{K'}(\Phi_x', \lambda)$. Hence it suffices to show that $f \in \cA^{K'} $ i.e.\ that $f$ respect the condition of moderate growth. 

Let $m$ be such that $q^{m-1}>|\lambda|^{-1}$. We are left to verify that $|f(c_{nx,w}')|<|c_{nx,w}'|^m$ for every $n \in \N$ and every $w \in \P^1(\FF_q)$, which shows the moderate condition. From above identities, $|f(c_{nx,w}')|$ equals either $\frac{q^n}{|\lambda|^n}$ or  $\frac{q^{n+2}}{|\lambda|^{n+2}}$, depending on $w \in \P^1(\Fq)$. We can suppose we are in the first case since the condition of moderate growth does not depend on any multiplicative constant. Since $|c_{nx,w}'|=q^n$, then 
\[\frac{|f(c_{nx,w}')|}{|c_{nx,w}'|^m} \leq \frac{q^n}{|\lambda|^nq^{nm}}=\frac{1}{(|\lambda|q^{m-1})^n} < 1, \]
where the last inequality follows from  the hypothesis on $m$.  
\end{proof}


\begin{df} \label{def-cusp} An automorphic form $f \in \cA$ is a cusp form (or cuspidal) if 
\[ \int_{U(F)\setminus U(\A)} f(ug)du = 0\]
for all $g \in G(\A)$ and all unipotent radicals subgroups $U$ of all standard parabolic subgroups $P$ of $G(\A)$. We denote the whole space of cusp forms by $\cA_{0}$.   If $f \in \cA^{K}$ is a cusp form, we call it an unramified cusp form and denote the whole space of unramified cusp forms by $\cA_{0}^{K}$. If $f \in \cA^{K'}$ is a cusp form, we call it a ramified cusp form (with respect to $K'$)  and denote the whole space of ramified cusp forms by $\cA_{0}^{K'}$. 
Finally, if $f \in \cA_{0}^{K'}$ is an $\Phi$-eigenform with eigenvalue $\lambda \in \C$, for some $\Phi \in \cH_{K'}$, we say $f$ is $\Phi$-eigenform cuspidal. We denote the whole space of $\Phi$-eigenforms cuspidal with eigenvalue $\lambda \in \C$ by $\cA_{0}^{K'}(\Phi, \lambda)$.  

\end{df}

\begin{cor} \label{cor-cuspramifiedforms} There are no nontrivial ramified (regarding $K'$) cuspidal $\Phi_x'$-eigenforms  for the projective line, i.e.\ $\cA_{0}^{K'}(\Phi_x ', \lambda) = \{0\}$, for all  $\lambda \in \C^{\times}$. 
\end{cor}

\begin{proof} By \cite[Thm. 3.2.2]{li-79} 
\[ \cA^{K'}(\Phi_x', \lambda) = \cE^{K'}(\Phi_x', \lambda) \oplus 
\cR^{K'}(\Phi_x', \lambda) \oplus \cA_{0}^{K'}(\Phi_x', \lambda) \]
where $\cE$ is the Eisenstein part of $\cA$  spanned by all Eisenstein series and their derivatives, and $\cR$ is the residual part of $\cA$ spanned by the residues of Eisenstein series and their derivatives. 

Let $B$ be the standard Borel subgroup of $G$ and $N$ its unipotent radical. 
The  \cite[Thm. 3.2.2]{li-79} still yields  
\[ \dim  \cE^{K'}(\Phi_x', \lambda) = \#\big(K/K'(Z N(\A)B(\Fq) \cap K)\big) \geq 1. \]
Comes from the above theorem that $\cA^{K'}(\Phi_x', \lambda) = \cE^{K'}(\Phi_x', \lambda)$ and therefore $\cA_{0}^{K'}(\Phi_x', \lambda)$ is trivial. 
    \end{proof}

    \begin{cor} \label{cor-trivicuspramifiedforms} There are no nontrivial ramified (regarding $K'$) cuspidal automorphic forms  for the projective line, i.e.\ $\cA_{0}^{K'} = \{0\}$. 
\end{cor}

\begin{proof}
   The space of cusp forms is a semi-simple $\GL_n(\A)$-module. Then, it is a direct sum of irreducible representations. Multiplicity One Theorem yields that every cusp form can be written as a sum of cusp eigenforms. Since $f \in \cA_{0}^{K'}$ is an $\cH_{K'}$-eigenform if it is $\Phi$-eigenform for every $\Phi \in \cH_{K'}$, the proof follows from the previous corollary. 
\end{proof}

For the unramified version of the above corollary, see \cite[Example 5.5.2]{oliver-thesis} for $\GL_2$ and \cite[Thm. 3.6]{alvarenga-pereira-23} for $\GL_n$, with every $n \geq 2$.

\subsection*{Graphs and eigenspaces of $\Phi_y'$ for $y\neq x$}

Now we are interested in the graphs and the eigenspaces of $\Phi_y'$ for a place $y \neq x$ of degree $d$. According to Theorems \ref{thm-ramifiedvertices} and \ref{thm-ramifiededges}, in order to describe $\cG_{\Phi_y',K'}$ we need to compute the representatives of the classes of $(p_{nx}\vartheta_{w})_x(\xi_{y,w'})_y$ in $G(F) \setminus G(\mathbb{A})/Z K'$ for every $w \in \P^1(\FF_q) $ and $w' \in \P^1(\kappa(y))$.

\begin{lemma} \label{lemma-uni}
Let $x,y \in |\P^1|$ of degrees $1$ and $d$, respectively. Then there exists a uniformizer $\pi_y$ for $y$ such that $v_z(\pi_y)=0$ for every $z \in |\P^1|$ with $z \neq x,y$.
\end{lemma}

\begin{proof}
Recall that $F$ is the function field of $\P^1$ i.e.\ the field of rational functions over $\Fq$. Let $\pi_y'$ be a uniformizer for $y$ and consider
$ D(\pi_y') := \sum_{z \in |\mathbb{P}^1|} v_z(\pi_y') z$
be the divisor associated with $\pi_y'$.  Let 
$ D(\pi_y') = D + D'$  with $D :=  y-d x$ and  $D' := d x + \sum_{z \neq y} v_z(\pi_y') z.$
Since, $\deg D' =0$ and the class number of $F$ is $1$, yields $D(\pi_y') - D = D(f)$ for some $f \in F.$ Moreover $v_y(f)=0$ since $D(f) = D'$. Thus, $f \in \mathcal{O}_{F,y}^{\times}$ and $\pi_y' f^{-1}$ is a uniformizer for $y$ with
\[D(\pi_y' f^{-1}) = D(\pi_y') - D(f) = D + D' - D' = y - dx.\]
Therefore for $\pi_y := \pi_y' f^{-1}$, we have that $\pi_y$ has a nontrivial valuation only in $y$ and $x.$
\end{proof}

From the previous lemma, 
we can suppose that $v_z(\pi_y)=0$ for every $ z \neq x,y$.
Then we have $\pi_x^d=\pi_y^{-1}$
 and $\big\{1,\pi_x^{-1}, \ldots ,\pi_x^{-d+1}\big\}$ is a $\FF_q$-bases for the $\FF_q$-vector space $\kappa(y)$. Hence, we can see $\xi_{y,w'}$ as a  $G(F)$ matrix using former basis. With this assumption, we can define $k' \in K'$ being $\id$ at places $x$ and $y$, and being $\xi_{y,w'} \in K_z$ for every other places $z$. Then, by multiplying by $(\xi_{y,w'})^{-1} \in G(F)$ on the left and by $k' \in K'$ on the right, we get to $(\xi_{y,w'}^{-1}p_{nx}\vartheta_w)_x$ and can use Proposition \ref{prop-ramifstrongaprox}.

Thus, we are left to find, according to Theorem \ref{thm-ramifiedvertices}, the standard representative classes of
\[
     \begin{pmatrix} \pi_x^d&b_0+b_1\pi_x+...+b_{d-1}\pi_x^{d-1} \\ &1 \end{pmatrix}p_{nx}\vartheta_w=\begin{pmatrix} \pi_x^{d-n}&b_0+b_1\pi_x+...+b_{d-1}\pi_x^{d-1} \\ &1 \end{pmatrix}\vartheta_w \]
     and
     \[ \begin{pmatrix} 1& \\ &\pi_x^d \end{pmatrix}p_{nx}\vartheta_w=\begin{pmatrix} \pi_x^{-n}& \\ &\pi_x^d \end{pmatrix}\vartheta_w
\]
in $\Gamma_x \setminus G_x/Z_xK_x'$, for every $b_0,b_1,...,b_{d-1} \in \FF_q$, $w \in \P^1(\FF_q)$ and  $n \in \N^{\times}$.


\begin{df} For every $k \in \N$, we define 
\[ \cN_k':=\big\{c_0'\big\}\cup \big\{c_{nx,w}' \;|\; n\in \{1,...,k\},~w\in\P^1(\FF_q) \big\}  \subset \mathrm{Vert}\; \cG_{\Phi_{y}',K'}. \]
We call $\mathfrak{N}_{d}' := \cN_{d}'-\{c_{dx,[1:0]}'\}$ the nucleus of $\cG_{\Phi_y ',K'}$, where $d = |y|$.
\end{df}

\begin{prop} \label{prop-3.16} Let $y \in |\P^1|$ of degree $d$ such that $y\neq x$. Then,
\begin{enumerate}[(i)]

        \item For every $n \geq d$ and $w \in \P^1(\FF_q)$, $c_{nx,w}$ has exactly two $\Phi_{y}'$-neighbors. Namely,  $c_{(n-d)x,w}$ with multiplicity $q^d$ and $c_{(n+d)x,w}$ with multiplicity $1$. \medskip
        
        \item  Every vertex $c_{kx,w}'$ in $\cN_{d}'-\{c_0'\}$ has one, and only one $\Phi_{y}'$-neighbor outside of $\cN_d'$. Namely,  $c_{(k+d)x,w}$ with multiplicity $1$. \medskip
        
        \item  The following equality holds, $m(c_0',c_{dx,[0:1]}')=q$, $m(c_0',c_{dx,[1:0]}')=1$ and every other $\Phi_{y}'$-neighbors of $c_0'$ are in $\cN_d'$.
 
\end{enumerate}
\end{prop}

\begin{proof}
    For \textit{(i)}, we first remark that if $n \geq d$,
{ \small    \begin{align*}
\begin{pmatrix} \pi_x^{d-n}&b_0+b_1\pi_x+...+b_{d-1}\pi_x^{d-1} \\ &1 \end{pmatrix}\vartheta_w &= p_{(n-d)x}\begin{pmatrix} 1&(b_0+b_1\pi_x+...+b_{d-1}\pi_x^{d-1})\pi_x^{n-d} \\ &1 \end{pmatrix}\vartheta_w\\
    & \sim p_{(n-d)x}\vartheta_w,  
\end{align*} }
for every $b_0, \ldots, b_{d-1} \in \Fq.$ Moreover 
\[ \begin{pmatrix} \pi_x^{-n}& \\ &\pi_x^d \end{pmatrix}\vartheta_w\sim p_{(n+d)x}\vartheta_w,\] 
which concludes \textit{(i)}. Observe that last equivalence holds for every $n \in \N$. 

For \textit{(ii)}, we use Lemma \ref{lemma2.1} and the fact that 

\[
\begin{pmatrix} \pi_x^{-n}& \\ &\pi_x^d \end{pmatrix}\vartheta_w\sim p_{(n+d)x}\vartheta_w,\] 
for every $n \in \N.$

To conclude \textit{(iii)}, in addition to the Lemma \ref{lemma2.1}, we just have to observe that  
\[\begin{pmatrix} \pi_x^{d}&b_0 \\ &1 \end{pmatrix}\sim c_{dx,[0:1]}'  \quad \text{ and } \quad  
\begin{pmatrix} 1& \\ &\pi_x^d \end{pmatrix}\sim c_{dx,[1:0]}.\]
for every $b_0 \in \Fq.$ 
\end{proof}

Based on the last proposition, in the following Figure \ref{fig6}, we draw the general picture of  $\mathcal{G}_{\Phi_y',K}$, for  $|y|=d$ and $y\neq x$.
The area delimited  by $\cN_{d-1}'$ might have some edges, even if nothing is drawn. We delimited the area given by the nucleus $\mathfrak{N}_{d}'$ as well. The next theorem explains its importance. 


\begin{center}
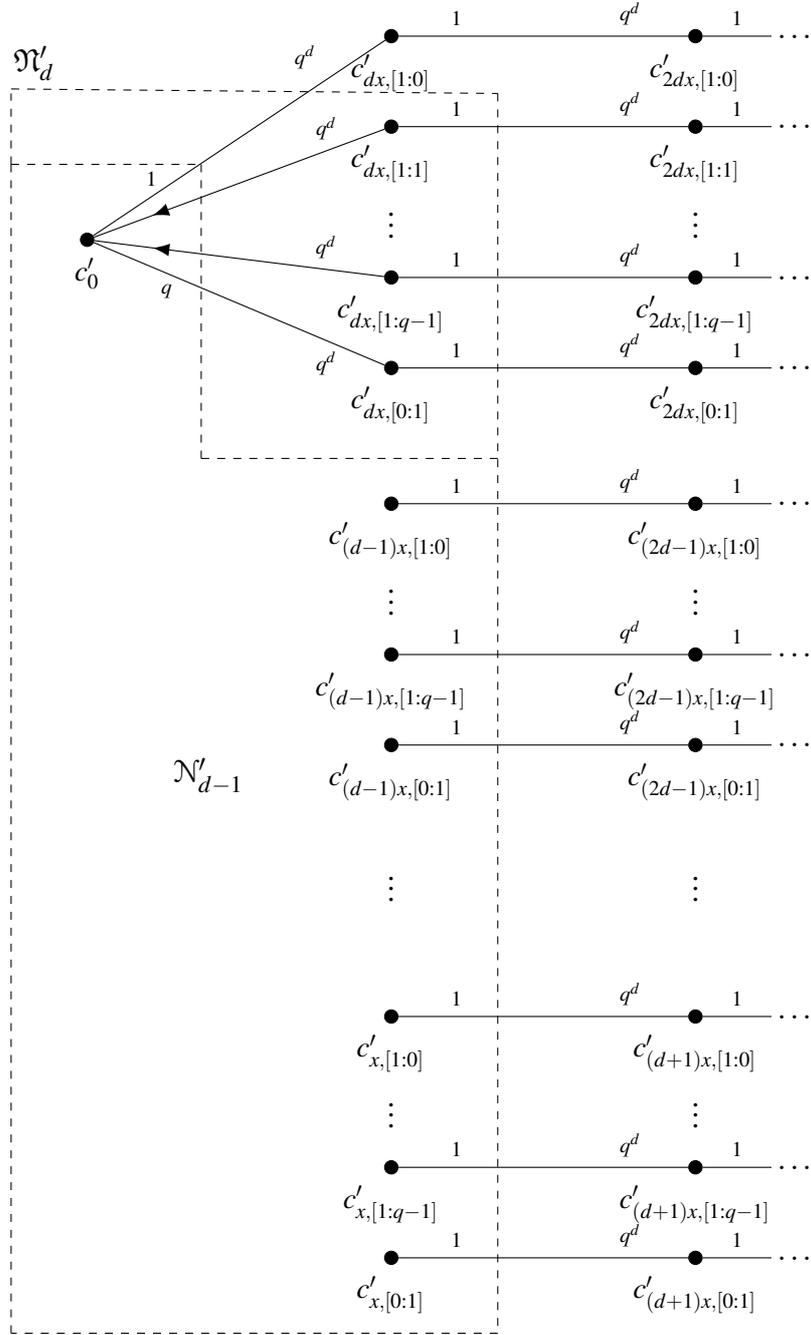
\begin{figure}[h]
\centering
\begin{minipage}[h]{\linewidth}
\[
  \beginpgfgraphicnamed{tikz/fig6}
  \begin{tikzpicture}[>=latex, scale=2]

        \vertex[circle,fill, minimum size = 5pt, label={below:\small $c_{0}'$}](c0) at (0,-.25) {};

      \vertex[circle,fill,minimum size = 5pt, label={below:\small $c_{dx,[1:0]}'$}](cdx10) at (2,1.1) {};
          \vertex[circle,fill,minimum size = 5pt, label={below:\small $c_{dx,[1:1]}'$}](cdx11) at (2,0.5) {}; 
      \node at (2,-.1) {$\vdots$};
      \vertex[circle,fill,minimum size = 5pt, label={below:\small $c_{dx,[1:q-1]}'$}](cdx1q-1) at (2,-.5){};
  \vertex[circle,fill,minimum size = 5pt, label={below :\small $c_{dx,[0:1]}'$}](cdx01) at (2,-1.1) {};

      \vertex[circle,fill,minimum size = 5pt, label={below:\small 
      $c_{(d-1)x,[1:0]}'$}](cd-1x10) at (2,-2) {}; 
      \node at (2,-2.6) {$\vdots$};
      \vertex[circle,fill,minimum size = 5pt, label={below:\small 
      $c_{(d-1)x,[1:q-1]}'$}](cd-1x1q-1) at (2,-3){};
  \vertex[circle,fill,minimum size = 5pt, label={below:\small 
  $c_{(d-1)x,[0:1]}'$}](cd-1x01) at (2,-3.6) {};

   \node at (2,-4.5) {$\vdots$};

      \vertex[circle,fill,minimum size = 5pt, label={below:\small 
      $c_{x,[1:0]}'$}](cx10) at (2,-5.4) {}; 
      \node at (2,-6) {$\vdots$};
      \vertex[circle,fill,minimum size = 5pt, label={below:\small 
      $c_{x,[1:q-1]}'$}](cx1q-1) at (2,-6.4){};
  \vertex[circle,fill,minimum size = 5pt, label={below:\small 
  $c_{x,[0:1]}'$}](cx01) at (2,-7) {};

      \vertex[circle,fill,minimum size = 5pt, label={below=-.1:\small $c_{2dx,[1:0]}'$}](c2dx10) at (4,1.1) {};
           \vertex[circle,fill,minimum size = 5pt, label={below=-.1:\small $c_{2dx,[1:1]}'$}](c2dx11) at (4,.5) {};
      \node at (4,-.1) {$\vdots$};
        \vertex[circle,fill,minimum size = 5pt, label={below:\small $c_{2dx,[1:q-1]}'$}](c2dx1q-1) at (4,-.5){};
  \vertex[circle,fill,minimum size = 5pt, label={below:\small $c_{2dx,[0:1]}'$}](c2dx01) at (4,-1.1) {};

      \vertex[circle,fill,minimum size = 5pt, label={below:\small 
      $c_{(2d-1)x,[1:0]}'$}](c2d-1x10) at (4,-2) {}; 
      \node at (4,-2.6) {$\vdots$};
      \vertex[circle,fill,minimum size = 5pt, label={below:\small 
      $c_{(2d-1)x,[1:q-1]}'$}](c2d-1x1q-1) at (4,-3){};
  \vertex[circle,fill,minimum size = 5pt, label={below:\small 
  $c_{(2d-1)x,[0:1]}'$}](c2d-1x01) at (4,-3.6) {};

   \node at (4,-4.5) {$\vdots$};

      \vertex[circle,fill,minimum size = 5pt, label={below:\small 
      $c_{(d+1)x,[1:0]}'$}](cd+1x10) at (4,-5.4) {}; 
      \node at (4,-6) {$\vdots$};
      \vertex[circle,fill,minimum size = 5pt, label={below:\small 
      $c_{(d+1)x,[1:q-1]}'$}](cd+1x1q-1) at (4,-6.4){};
  \vertex[circle,fill,minimum size = 5pt, label={below:\small 
  $c_{(d+1)x,[0:1]}'$}](cd+1x01) at (4,-7) {};


  \path[-,font=\scriptsize]
   (c0) edge[-=0.8] node[pos=0.2,auto,above]{\tiny $1$}  
   node[pos=0.8,auto,black]{\tiny $q^d$} (cdx10) ;
 
 \path[-,font=\scriptsize] (c0) edge[-=0.8] node[pos=0.25,auto,below]{\tiny $q$} node[pos=0.8,below,black]{\tiny $q^d$}  (cdx01) ;

   \path[-,font=\scriptsize]
   (cdx10) edge[-=0.8] node[pos=0.2,auto,black]{\tiny $1$} 
    node[pos=0.8,auto,black]{\tiny $q^d$}  (c2dx10) ;
     \path[-,font=\scriptsize]
   (cdx11) edge[->-=0.8] node[pos=0.2,above,black]{\tiny $q^d$} 
     (c0) ;
        \path[-,font=\scriptsize]
   (cdx11) edge[-=0.8] node[pos=0.2,auto,black]{\tiny $1$} 
    node[pos=0.8,auto,black]{\tiny $q^d$} (c2dx11);
   \path[-,font=\scriptsize]
   (cdx1q-1) edge[->-=0.8] node[pos=0.2,above,black]{\tiny $q^d$} 
     (c0) ;
      \path[-,font=\scriptsize]
   (cdx1q-1) edge[-=0.8] node[pos=0.2,auto,black]{\tiny $1$} 
    node[pos=0.8,auto,black]{\tiny $q^d$}  (c2dx1q-1) ;
      \path[-,font=\scriptsize]
   (cdx01) edge[-=0.8] node[pos=0.2,auto,black]{\tiny $1$} 
    node[pos=0.8,auto,black]{\tiny $q^d$}  (c2dx01) ;

   \path[-,font=\scriptsize]
   (cd-1x10) edge[-=0.8] node[pos=0.2,auto,black]{\tiny $1$} 
    node[pos=0.8,auto,black]{\tiny $q^d$}  (c2d-1x10) ;
      \path[-,font=\scriptsize]
   (cd-1x1q-1) edge[-=0.8] node[pos=0.2,auto,black]{\tiny $1$} 
    node[pos=0.8,auto,black]{\tiny $q^d$}  (c2d-1x1q-1) ;
      \path[-,font=\scriptsize]
   (cd-1x01) edge[-=0.8] node[pos=0.2,auto,black]{\tiny $1$} 
    node[pos=0.8,auto,black]{\tiny $q^d$}  (c2d-1x01) ;

       \path[-,font=\scriptsize]
   (cx10) edge[-=0.8] node[pos=0.2,auto,black]{\tiny $1$} 
    node[pos=0.8,auto,black]{\tiny $q^d$}  (cd+1x10) ;
      \path[-,font=\scriptsize]
   (cx1q-1) edge[-=0.8] node[pos=0.2,auto,black]{\tiny $1$} 
    node[pos=0.8,auto,black]{\tiny $q^d$}  (cd+1x1q-1) ;
      \path[-,font=\scriptsize]
   (cx01) edge[-=0.8] node[pos=0.2,auto,black]{\tiny $1$} 
    node[pos=0.8,auto,black]{\tiny $q^d$}  (cd+1x01) ;

   \path[-,font=\scriptsize]
   (c2dx10) edge[-=0.8] node[pos=0.5,auto,black]{\tiny $1$}  (4.5,1.1) ;
      \path[-,font=\scriptsize]
   (c2dx11) edge[-=0.8] node[pos=0.5,auto,black]{\tiny $1$}  (4.5,0.5) ;
      \path[-,font=\scriptsize]
   (c2dx1q-1) edge[-=0.8] node[pos=0.5,auto,black]{\tiny $1$} (4.5,-0.5) ;
      \path[-,font=\scriptsize]
   (c2dx01) edge[-=0.8] node[pos=0.5,auto,black]{\tiny $1$} (4.5,-1.1) ;

\node at (4.65,1.1) {$\ldots$};
\node at (4.65,0.5) {$\ldots$};
\node at (4.65,-.5) {$\ldots$};
\node at (4.65,-1.1) {$\ldots$};


   \path[-,font=\scriptsize]
   (c2d-1x10) edge[-=0.8] node[pos=0.5,auto,black]{\tiny $1$}  (4.5,-2) ;
      \path[-,font=\scriptsize]
   (c2d-1x1q-1) edge[-=0.8] node[pos=0.5,auto,black]{\tiny $1$} (4.5,-3) ;
      \path[-,font=\scriptsize]
   (c2d-1x01) edge[-=0.8] node[pos=0.5,auto,black]{\tiny $1$} (4.5,-3.6) ;

\node at (4.65,-2) {$\ldots$};
\node at (4.65,-3) {$\ldots$};
\node at (4.65,-3.6) {$\ldots$};


   \path[-,font=\scriptsize]
   (cd+1x10) edge[-=0.8] node[pos=0.5,auto,black]{\tiny $1$}  (4.5,-5.4) ;
      \path[-,font=\scriptsize]
   (cd+1x1q-1) edge[-=0.8] node[pos=0.5,auto,black]{\tiny $1$} (4.5,-6.4) ;
      \path[-,font=\scriptsize]
   (cd+1x01) edge[-=0.8] node[pos=0.5,auto,black]{\tiny $1$} (4.5,-7) ;

\node at (4.65,-5.4) {$\ldots$};
\node at (4.65,-6.4) {$\ldots$};
\node at (4.65,-7) {$\ldots$};


\node[auto,above] at (.8,-4) { $\mathcal{N}_{d-1}'$};
\node[auto,above] at (-0.35,.73) { $\mathfrak{N}_d'$};

\draw[dashed]  (2.7,0.72)  -- (2.7,-7.5); 
\draw[dashed]  (2.7,-7.5)  -- (-0.5,-7.5); 
\draw[dashed]  (-.5,-7.5) -- (-.5,.25); 
\draw[dashed]  (-.5,.25) -- (-.5,.72);
\draw[dashed]  (-.5,.75) -- (2.7,0.72); 
\draw[dashed] (-.5,.25) -- (.75, .25);
\draw[dashed]  (.75, .25) -- (0.75,-1.7);
\draw[dashed]  (.75, -1.7) -- (2.7,-1.7);


   \end{tikzpicture}
 \endpgfgraphicnamed
\]
\end{minipage} 
 \caption{The general graph $\mathcal{G}_{\Phi_y',K}$, for  $|y|=d$ and $y\neq x$.}
  \label{fig6}
\end{figure}

\end{center}

\begin{thm} \label{thm-mainramified2}
    Let $y \in |\P^1$ of degree $d$ and $\lambda \in \C^{\times}$. 
The map 
\[\begin{array}{cccc}
       \cA^{K'}(\Phi_y', \lambda) & \longrightarrow & \bigoplus_{\mathfrak{N}_{d}'}\C \\[4pt]
   f & \longmapsto & (f(g))_{g \in \mathfrak{N}_{d}'} \\
\end{array}\] is an isomorphism of $\C$-vector spaces. In particular, $\dim \cA^{K'}(\Phi_y', \lambda) = d(q+1)$.

\end{thm}

\begin{proof} 
 Let $k\geq 0$ and $f,f' \in \cA^{K'}(\Phi_y', \lambda)$ be such that $f|_{\cN_{(k+1)d}'}=f'|_{\cN_{(k+1)d}'}$. 
 
 Let $w \in \P^1$ and $l \in \{1,...,d\}$. According to the last proposition, the $\Phi_y'$-neighbors of $c_{(kd+l)x,w}'$ are all in $\cN_{kd+l}'$, except for $c_{((k+1)d+l)x,w}'$. Thus, since $f$ is an $\Phi_y'$-eigenform, 
\[\lambda f(c_{(kd+l)x,w}')=\Phi_y(f)(c_{(kd+l)x,w}')= f(c_{((k+1)d+l)x,w}')-\sum_{g \in \cN_{kd+l}'}m(c_{(kd+l)x,w}',g)f(g)\]
 and the same holds for $f'$. 
By assumption $f|_{\cN_{(k+1)d}'}=f'|_{\cN_{(k+1)d}'}$, which implies  $f|_{\cN_{kd+l}'}=f'|_{\cN_{kd+l}'}$. Thus, $f(c_{((k+1)d+l)x})=f'(c_{((k+1)d+l)x})$ since the graph $\cG_{\Phi_{y'},K'}$ does not depend on the automorphic form under consideration.
Everything above is true for every $l \in \{1,...,d\}$ and for every $w \in \P^1$, thus $f|_{\cN_{(k+2)d}'}=f'|_{\cN_{(k+2)d}'}$. Therefore, we conclude by induction that $f|_{\cN_{d}'}=f'|_{\cN_{d}}$ yields $f=f'$.
 
 Moreover, according to the Lemma \ref{lemma2.1}, 
 \[ \lambda f(c_0')=\Phi_y'(f)(c_0')=qf(c_{dx,[0:1]}')+f(c_{dx,[1:0]}')+\sum_{g \in \cN_{d-1}'}m(c_0',g),\]
 hence $f|_{\mathfrak{N}_{d}'}=f'|_{\mathfrak{N}_{d}'}$ yields $f=f'$, which is what we needed for the injectivity.

For the surjectivity, we must show that we can find $f \in \cA^{K'}(\Phi_y', \lambda)$ with fixed values on $\mathfrak{N}_{d}'$.
Hence, suppose that $f$ is a function defined on $\mathfrak{N}_{d}'$.  The goal is to extend $f$ to $G(F) \setminus G(\mathbb{A})/Z K'$ as an $\Phi_{y}'$-eigenform.
By  Proposition \ref{prop-3.16}, if we define
\[f(c_{dx,[1:0]}'):=\lambda f(c_0')-qf(c_{dx,[0:1]}')-\sum_{g \in \cN_{d-1}'}m(c_0',g)\]
then $\Phi_y'(f)(c_{0}')=\lambda f(c_{0}')$. At this point $f$ is defined on $\cN_d'$.
Next, for every $k \in \{1,...,d-1\}$ and $w \in \P^1$, we define 
\[ f(c_{(k+d)x,w}'):=\lambda f(c_{kx,w}')-\sum_{g \in \cN_{d-1}'}m(c_0',g)\]
which yields $\Phi_y'(f)(c_{kx,w}')=\lambda f(c_{kx,w}')$ by  Proposition \ref{prop-3.16}.

If $n \geq 2d$ we just define by induction $f(c_{nx,w}'):=\lambda f(c_{(n-d)x,w}')-qf(c_{(n-2d)x,w}')$.
The  moderate growth condition is satisfied and $f$ is, by construction, an element of $\cA^{K'}(\Phi_y', \lambda)$. This concludes the proof. 
\end{proof}

Followingly, we describe the full graphs $\cG_{\Phi_y',K'}$ for $|y|=1,2,3$ by using methods explained in Remark \ref{rem-calculation}.


\begin{center}

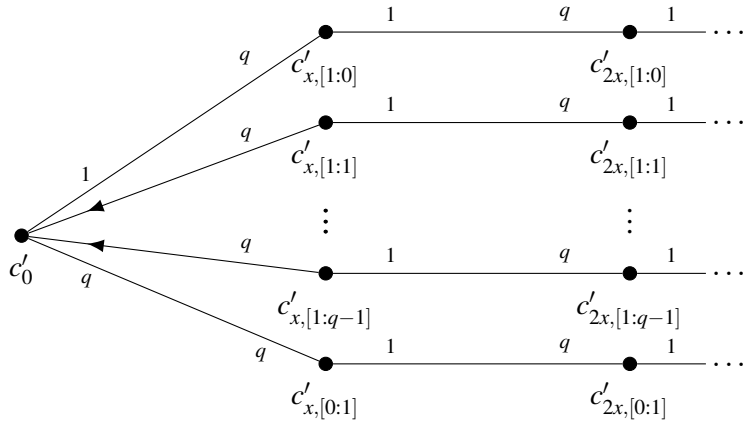
\begin{figure}[h]
\centering
\begin{minipage}[h]{\linewidth}
\[
  \beginpgfgraphicnamed{tikz/fig7}
  \begin{tikzpicture}[>=latex, scale=2]

        \vertex[circle,fill, minimum size = 5pt, label={below:\small $c_{0}'$}](c0) at (0,-.25) {};

      \vertex[circle,fill,minimum size = 5pt, label={below:\small $c_{x,[1:0]}'$}](cx10) at (2,1.1) {};
          \vertex[circle,fill,minimum size = 5pt, label={below:\small $c_{x,[1:1]}'$}](cx11) at (2,0.5) {}; 
      \node at (2,-.1) {$\vdots$};
      \vertex[circle,fill,minimum size = 5pt, label={below:\small $c_{x,[1:q-1]}'$}](cx1q-1) at (2,-.5){};
  \vertex[circle,fill,minimum size = 5pt, label={below :\small $c_{x,[0:1]}'$}](cx01) at (2,-1.1) {};
      \vertex[circle,fill,minimum size = 5pt, label={below:\small $c_{2x,[1:0]}'$}](c2x10) at (4,1.1) {};
          \vertex[circle,fill,minimum size = 5pt, label={below:\small $c_{2x,[1:1]}'$}](c2x11) at (4,0.5) {}; 
      \node at (2,-.1) {$\vdots$};
      \vertex[circle,fill,minimum size = 5pt, label={below:\small $c_{2x,[1:q-1]}'$}](c2x1q-1) at (4,-.5){};
  \vertex[circle,fill,minimum size = 5pt, label={below :\small $c_{2x,[0:1]}'$}](c2x01) at (4,-1.1) {};
   
\node at (2,-.1) {$\vdots$};


  \path[-,font=\scriptsize]
   (c0) edge[-=0.8] node[pos=0.2,auto,above]{\tiny $1$} 
    node[pos=0.8,auto,black]{\tiny $q$}  (cx10) ;
    \path[-,font=\scriptsize] (cx11) edge[->-=0.8] node[pos=0.25,auto,above]{\tiny $q$} 
    node[pos=0.8,auto,black]{}  (c0) ;
 \path[-,font=\scriptsize] (cx1q-1) edge[->-=0.8] node[pos=0.25,auto,above]{\tiny $q$} 
    node[pos=0.8,auto,black]{}  (c0) ;
  \path[-,font=\scriptsize]  (c0) edge[-=0.8] node[pos=0.2,auto,below]{\tiny $q$} 
    node[pos=0.8,below,black]{\tiny $q$}  (cx01) ;

   \path[-,font=\scriptsize]
   (cx10) edge[-=0.8] node[pos=0.2,auto,black]{\tiny $1$} 
    node[pos=0.8,auto,black]{\tiny $q$}  (c2x10) ;
       \path[-,font=\scriptsize]
   (cx11) edge[-=0.8] node[pos=0.2,auto,black]{\tiny $1$} 
    node[pos=0.8,auto,black]{\tiny $q$}  (c2x11) ;
      \path[-,font=\scriptsize]
   (cx1q-1) edge[-=0.8] node[pos=0.2,auto,black]{\tiny $1$} 
    node[pos=0.8,auto,black]{\tiny $q$}  (c2x1q-1) ;
      \path[-,font=\scriptsize]
   (cx01) edge[-=0.8] node[pos=0.2,auto,black]{\tiny $1$} 
    node[pos=0.8,auto,black]{\tiny $q$}  (c2x01) ;

   \path[-,font=\scriptsize]
   (c2x10) edge[-=0.8] node[pos=0.5,auto,black]{\tiny $1$}  (4.5,1.1) ;
     \path[-,font=\scriptsize]
   (c2x11) edge[-=0.8] node[pos=0.5,auto,black]{\tiny $1$}  (4.5,0.5) ; 
      \path[-,font=\scriptsize]
   (c2x1q-1) edge[-=0.8] node[pos=0.5,auto,black]{\tiny $1$} (4.5,-0.5) ;
      \path[-,font=\scriptsize]
   (c2x01) edge[-=0.8] node[pos=0.5,auto,black]{\tiny $1$} (4.5,-1.1) ;
    
\node at (4.65,1.1) {$\ldots$};
\node at (4.65,0.5) {$\ldots$};
\node at (4.65,-.5) {$\ldots$};
\node at (4.65,-1.1) {$\ldots$};

\node at (4,-.1) {$\vdots$};

   \end{tikzpicture}
 \endpgfgraphicnamed
\]
\end{minipage} 
 \caption{The graph $\mathcal{G}_{\Phi_y',K}$ for $|y|=1$.}
  \label{fig7}
\end{figure}

\end{center}


\begin{center}

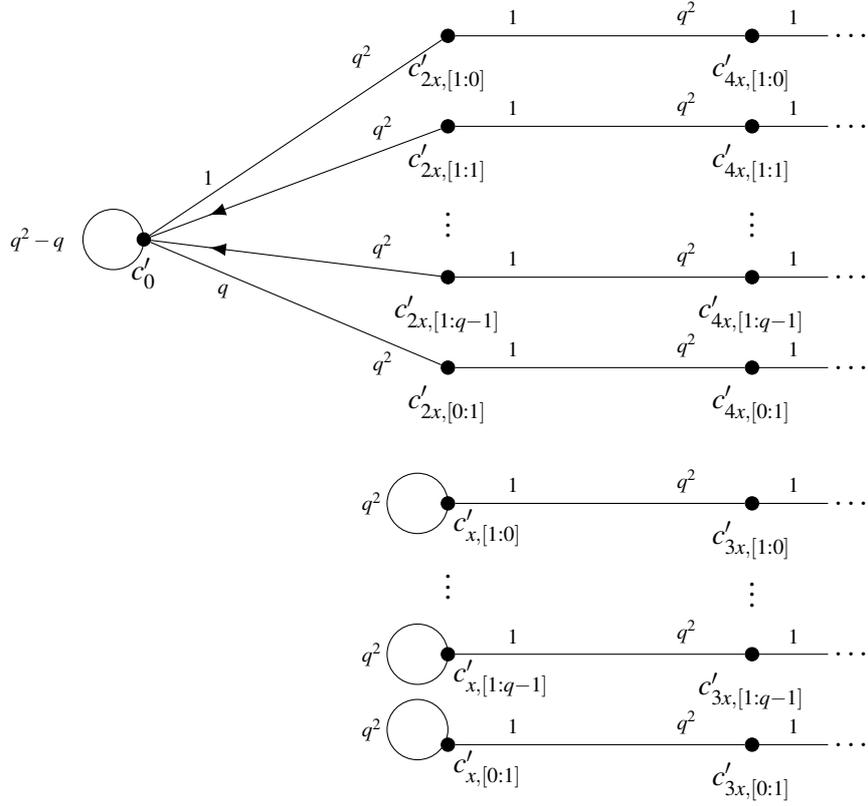
\begin{figure}[h]
\centering
\begin{minipage}[h]{\linewidth}
\[
  \beginpgfgraphicnamed{tikz/fig8}
  \begin{tikzpicture}[>=latex, scale=2]

        \vertex[circle,fill, minimum size = 5pt, label={below:\small $c_{0}'$}](c0) at (0,-.25) {};

      \vertex[circle,fill,minimum size = 5pt, label={below:\small $c_{2x,[1:0]}'$}](c2x10) at (2,1.1) {};

      \vertex[circle,fill,minimum size = 5pt, label={below:\small $c_{2x,[1:1]}'$}](c2x11) at (2,.5) {};
      \node at (2,-.1) {$\vdots$};
        \vertex[circle,fill,minimum size = 5pt, label={below, right:\small $c_{2x,[1:q-1]}'$}](c2x1q-1) at (2,-.5){};
  \vertex[circle,fill,minimum size = 5pt, label={below, right:\small $c_{2x,[0:1]}'$}](c2x01) at (2,-1.1) {};

      \vertex[circle,fill,minimum size = 5pt, label={[label distance=-0.2cm] below right: \small $c_{x,[1:0]}'$}](cx10) at (2,-2) {};
      \node at (2,-2.5) {$\vdots$};
        \vertex[circle,fill,minimum size = 5pt, label={[label distance=-0.2cm] below right:\small $c_{x,[1:q-1]}'$}](cx1q-1) at (2,-3){};
  \vertex[circle,fill,minimum size = 5pt, label={[label distance=-0.2cm] below right:\small $c_{x,[0:1]}'$}](cx01) at (2,-3.6) {};
  
      \vertex[circle,fill,minimum size = 5pt, label={below=-.1:\small $c_{4x,[1:0]}'$}](c4x10) at (4,1.1) {};
        \vertex[circle,fill,minimum size = 5pt, label={below=-.1:\small $c_{4x,[1:1]}'$}](c4x11) at (4,0.5) {};
       \node at (4,-.1) {$\vdots$};
        \vertex[circle,fill,minimum size = 5pt, label={below:\small $c_{4x,[1:q-1]}'$}](c4x1q-1) at (4,-.5){};
  \vertex[circle,fill,minimum size = 5pt, label={below:\small $c_{4x,[0:1]}'$}](c4x01) at (4,-1.1) {};

        \vertex[circle,fill,minimum size = 5pt, label={below=-.1:\small $c_{3x,[1:0]}'$}](c3x10) at (4,-2) {};
           \node at (4,-2.55) {$\vdots$};
        \vertex[circle,fill,minimum size = 5pt, label={below:\small $c_{3x,[1:q-1]}'$}](c3x1q-1) at (4,-3){};
  \vertex[circle,fill,minimum size = 5pt, label={below:\small $c_{3x,[0:1]}'$}](c3x01) at (4,-3.6) {};


  \path[-,font=\scriptsize]
   (c0) edge[-=0.8] node[pos=0.2,auto,above]{\tiny $1$} 
    node[pos=0.8,auto,black]{\tiny $q^2$}  (c2x10) ;

 \path[-,font=\scriptsize] (c0) edge[-=0.8] node[pos=0.25,below,black]{\tiny $q$} node[pos=0.8,below,black]{\tiny $q^2$}  (c2x01) ;

\draw[black] (-0.2,-0.25) circle  (0.2cm) node at (-0.7,-0.25) {\tiny $q^2-q$};

   \path[-,font=\scriptsize]
   (c2x10) edge[-=0.8] node[pos=0.2,auto,black]{\tiny $1$} 
    node[pos=0.8,auto,black]{\tiny $q^2$}  (c4x10) ;

       \path[-,font=\scriptsize]
   (c2x11) edge[->-=0.8] 
    node[pos=0.2,above,black]{\tiny $q^2$}  (c0) ;

           \path[-,font=\scriptsize]
   (c2x11) edge[-=0.8] node[pos=0.2,auto,black]{\tiny $1$} 
    node[pos=0.8,above,black]{\tiny $q^2$}   (c4x11) ;
     
      \path[-,font=\scriptsize]
   (c2x1q-1) edge[-=0.8] node[pos=0.2,auto,black]{\tiny $1$} 
    node[pos=0.8,above,black]{\tiny $q^2$}  (c4x1q-1) ;

          \path[-,font=\scriptsize]
   (c2x1q-1) edge[->-=0.8] node[pos=0.2,above,black]{\tiny $q^2$}  (c0) ;
    
      \path[-,font=\scriptsize]
   (c2x01) edge[-=0.8] node[pos=0.2,above,black]{\tiny $1$} 
    node[pos=0.8,above,black]{\tiny $q^2$}  (c4x01) ;


   \path[-,font=\scriptsize]
   (cx10) edge[-=0.8] node[pos=0.2,auto,black]{\tiny $1$} 
    node[pos=0.8,auto,black]{\tiny $q^2$}  (c3x10) ;

 \draw[black] (1.8,-2) circle  (0.2cm) node at (1.5,-2) {\tiny $q^2$};
     
      \path[-,font=\scriptsize]
   (cx1q-1) edge[-=0.8] node[pos=0.2,auto,black]{\tiny $1$} 
    node[pos=0.8,auto,black]{\tiny $q^2$}  (c3x1q-1) ;

 \draw[black] (1.8,-3) circle  (0.2cm) node at (1.5,-3) {\tiny $q^2$};
    
      \path[-,font=\scriptsize]
   (cx01) edge[-=0.8] node[pos=0.2,auto,black]{\tiny $1$} 
    node[pos=0.8,auto,black]{\tiny $q^2$}  (c3x01) ;

 \draw[black] (1.8,-3.5) circle  (0.2cm) node at (1.5,-3.5) {\tiny $q^2$};

   \path[-,font=\scriptsize]
   (c4x10) edge[-=0.8] node[pos=0.5,auto,black]{\tiny $1$}  (4.5,1.1) ;
      \path[-,font=\scriptsize]
   (c4x11) edge[-=0.8] node[pos=0.5,auto,black]{\tiny $1$}  (4.5,0.5) ;
  \path[-,font=\scriptsize]
   (c4x1q-1) edge[-=0.8] node[pos=0.5,auto,black]{\tiny $1$} (4.5,-0.5) ;
     \path[-,font=\scriptsize]
   (c4x01) edge[-=0.8] node[pos=0.5,auto,black]{\tiny $1$} (4.5,-1.1) ;

   \path[-,font=\scriptsize]
   (c3x10) edge[-=0.8] node[pos=0.5,auto,black]{\tiny $1$}  (4.5,-2) ;
  \path[-,font=\scriptsize]
   (c3x1q-1) edge[-=0.8] node[pos=0.5,auto,black]{\tiny $1$} (4.5,-3) ;
     \path[-,font=\scriptsize]
   (c3x01) edge[-=0.8] node[pos=0.5,auto,black]{\tiny $1$} (4.5,-3.6) ;    
\node at (4.65,0.5) {$\ldots$};
\node at (4.65,1.1) {$\ldots$};
\node at (4.65,-.5) {$\ldots$};
\node at (4.65,-1.1) {$\ldots$};

\node at (4.65,-2) {$\ldots$};
\node at (4.65,-3) {$\ldots$};
\node at (4.65,-3.6) {$\ldots$};

   \end{tikzpicture}
 \endpgfgraphicnamed
\]
\end{minipage} 
 \caption{The graph $\mathcal{G}_{\Phi_y',K}$ for $|y|=2$.}
  \label{fig8}
\end{figure}

\end{center}


\begin{center}
\begin{figure}[h]
\centering
\begin{minipage}[h]{\linewidth}
\[
  \beginpgfgraphicnamed{tikz/fig9}
  \begin{tikzpicture}[>=latex, scale=2]

        \vertex[circle,fill, minimum size = 5pt, label={below:\small $c_{0}'$}](c0) at (0,-.25) {};

      \vertex[circle,fill,minimum size = 5pt, label={below:\small $c_{3x,[1:0]}'$}](c3x10) at (-2,1.1) {};

      \vertex[circle,fill,minimum size = 5pt, label={below:\small $c_{3x,[1:1]}'$}](c3x11) at (-2,.5) {};
      \node at (2,-.1) {$\vdots$};
        \vertex[circle,fill,minimum size = 5pt, label={below, right:\small $c_{3x,[1:q-1]}'$}](c3x1q-1) at (-2,-.5){};
  \vertex[circle,fill,minimum size = 5pt, label={below, right:\small $c_{3x,[0:1]}'$}](c3x01) at (-2,-1.1) {};

     
      \vertex[circle,fill,minimum size = 5pt, label={below right:\small $c_{x,[1:0]}'$}](cx10) at (2,1.1) {};

      \vertex[circle,fill,minimum size = 5pt, label={below right:\small $c_{x,[1:1]}'$}](cx11) at (2,.5) {};
      \node at (2,-.1) {$\vdots$};
        \vertex[circle,fill,minimum size = 5pt, label={below right:\small $c_{x,[1:q-1]}'$}](cx1q-1) at (2,-.5){};
  \vertex[circle,fill,minimum size = 5pt, label={below  right:\small $c_{x,[0:1]}'$}](cx01) at (2,-1.1) {};

      \vertex[circle,fill,minimum size = 5pt, label={below right :\small $c_{2x,[1:0]}'$}](c2x10) at (2,-2) {};

      \vertex[circle,fill,minimum size = 5pt, label={below right:\small $c_{2x,[1:1]}'$}](c2x11) at (2,-2.6) {};
      \node at (2,-3.2) {$\vdots$};
        \vertex[circle,fill,minimum size = 5pt, label={below right:\small $c_{2x,[1:q-1]}'$}](c2x1q-1) at (2,-3.6){};
  \vertex[circle,fill,minimum size = 5pt, label={below  right:\small $c_{2x,[0:1]}'$}](c2x01) at (2,-4.2) {};

      \vertex[circle,fill,minimum size = 5pt, label={below=-.1:\small $c_{4x,[1:0]}'$}](c4x10) at (4,1.1) {};
        \vertex[circle,fill,minimum size = 5pt, label={below=-.1:\small $c_{4x,[1:1]}'$}](c4x11) at (4,0.5) {};
       \node at (4,-.1) {$\vdots$};
        \vertex[circle,fill,minimum size = 5pt, label={below:\small $c_{4x,[1:q-1]}'$}](c4x1q-1) at (4,-.5){};
  \vertex[circle,fill,minimum size = 5pt, label={below:\small $c_{4x,[0:1]}'$}](c4x01) at (4,-1.1) {};

        \vertex[circle,fill,minimum size = 5pt, label={below=-.1:\small $c_{5x,[1:0]}'$}](c5x10) at (4,-2) {};
           \vertex[circle,fill,minimum size = 5pt, label={below=-.1:\small $c_{5x,[1:1]}'$}](c5x11) at (4,-2.6) {};
           \node at (4,-3.2) {$\vdots$};
        \vertex[circle,fill,minimum size = 5pt, label={below:\small $c_{5x,[1:q-1]}'$}](c5x1q-1) at (4,-3.6){};
  \vertex[circle,fill,minimum size = 5pt, label={below:\small $c_{5x,[0:1]}'$}](c5x01) at (4,-4.2) {};


  \path[-,font=\scriptsize]
   (c0) edge[-=0.8] node[pos=0.2,auto,above]{\tiny $1$} 
    node[pos=0.75,above,black]{\tiny $q^3$}  (c3x10) ;

 \path[-,font=\scriptsize] (c0) edge[-=0.8] node[pos=0.25,below,black]{\tiny $q$} node[pos=0.8,below,black]{\tiny $q^3$}  (c3x01) ;

 \path[-,font=\scriptsize] (c0) edge[-=0.8] node[pos=0.2,above=.1,black]{\tiny $q^2-q$} node[pos=0.8,above=.1,black]{\tiny $q^3-q^2$}  (cx10) ;

  \path[-,font=\scriptsize] (c0) edge[-=0.8] node[pos=0.25,below=.1,black]{\tiny $q^3-q^2$} node[pos=0.8,below,black]{\tiny $q^3-q^2$}  (cx01) ;


   \path[-,font=\scriptsize]
   (c3x11) edge[->-=0.8] node[pos=0.2,above=.01,black]{\tiny $q^3$} (c0);
   \path[-,font=\scriptsize]
   (c3x1q-1) edge[->-=0.8] node[pos=0.2,above=.01,black]{\tiny $q^3$} (c0);

      \path[-,font=\scriptsize]
   (c3x10) edge[-=0.8] node[pos=0.5,above,black]{\tiny $1$}  (-2.5,1.1) ;
      \path[-,font=\scriptsize]
   (c3x11) edge[-=0.8] node[pos=0.5,above,black]{\tiny $1$}  (-2.5,0.5) ;
  \path[-,font=\scriptsize]
   (c3x1q-1) edge[-=0.8] node[pos=0.5,above,black]{\tiny $1$} (-2.5,-0.5) ;
     \path[-,font=\scriptsize]
   (c3x01) edge[-=0.8] node[pos=0.5,above,black]{\tiny $1$} (-2.5,-1.1) ;

\node at (-2.65,1.1) {$\ldots$};
\node at (-2.65,0.5) {$\ldots$};
\node at (-2.65,-0.5) {$\ldots$};
\node at (-2.65,-1.1) {$\ldots$};

   \path[-,font=\scriptsize]
   (cx11) edge[->-=0.8] node[pos=0.15,below,black]{\tiny $q^3-q^2$}  (c0) ;
      \path[-,font=\scriptsize]
   (cx1q-1) edge[->-=0.8] node[pos=0.18,above=-.1,black]{\tiny $q^3-q^2$}  (c0) ;

\path[-,font=\scriptsize]
   (cx10) edge[-=0.8] node[pos=0.2,auto,black]{\tiny $1$} 
    node[pos=0.8,above,black]{\tiny $q^3$}   (c4x10) ;
\path[-,font=\scriptsize]
   (cx11) edge[-=0.8] node[pos=0.2,auto,black]{\tiny $1$} 
    node[pos=0.8,above,black]{\tiny $q^3$}   (c4x11) ;
     \path[-,font=\scriptsize]
   (cx1q-1) edge[-=0.8] node[pos=0.2,auto,black]{\tiny $1$} 
    node[pos=0.8,above,black]{\tiny $q^3$}  (c4x1q-1) ;
    \path[-,font=\scriptsize]
   (cx01) edge[-=0.8] node[pos=0.2,auto,black]{\tiny $1$} 
    node[pos=0.8,above,black]{\tiny $q^3$}   (c4x01) ;

\path[-,font=\scriptsize]
   (c2x10) edge[-=0.8] node[pos=0.2,auto,black]{\tiny $1$} 
    node[pos=0.8,above,black]{\tiny $q^3$}   (c5x10) ;
\path[-,font=\scriptsize]
   (c2x11) edge[-=0.8] node[pos=0.2,auto,black]{\tiny $1$} 
    node[pos=0.8,above,black]{\tiny $q^3$}   (c5x11) ;
     \path[-,font=\scriptsize]
   (c2x1q-1) edge[-=0.8] node[pos=0.2,auto,black]{\tiny $1$} 
    node[pos=0.8,above,black]{\tiny $q^3$}  (c5x1q-1) ;
    \path[-,font=\scriptsize]
   (c2x01) edge[-=0.8] node[pos=0.2,auto,black]{\tiny $1$} 
    node[pos=0.8,above,black]{\tiny $q^3$}   (c5x01) ;

\path[-,font=\scriptsize]
(c2x10)  edge[-=0.85, bend left] node[pos=0.2,left,black]{$q^3$}
 node[pos=0.8,left,black]{\tiny $q^2$} (cx01);

\path[-,font=\scriptsize]
(c2x01)  edge[-=0.85, bend left=45] node[pos=0.1,left,black] {$q^3$} node[pos=0.9,left=-.1,black]{\tiny $q^2$} (cx10);

\path[-,font=\scriptsize]
(1.95,-3.25)  edge[dashed, bend left] node[pos=0.2,left,black] {} node[pos=0.8,left,black]{} (1.95,-.15);


\path[-,font=\scriptsize]
(c4x10) edge[-=0.8] node[pos=0.2,left,black] {} node[pos=0.8,left,black]{} (4.5,1.1);

\path[-,font=\scriptsize]
(c4x11) edge[-=0.8] node[pos=0.2,left,black] {} node[pos=0.8,left,black]{} (4.5,.5);

\path[-,font=\scriptsize]
(c4x1q-1) edge[-=0.8] node[pos=0.2,left,black] {} node[pos=0.8,left,black]{} (4.5,-.5);

\path[-,font=\scriptsize]
(c4x01) edge[-=0.8] node[pos=0.2,left,black] {} node[pos=0.8,left,black]{} (4.5,-1.1);


\path[-,font=\scriptsize]
(c5x10) edge[-=0.8] node[pos=0.2,left,black] {} node[pos=0.8,left,black]{} (4.5,-2);

\path[-,font=\scriptsize]
(c5x11) edge[-=0.8] node[pos=0.2,left,black] {} node[pos=0.8,left,black]{} (4.5,-2.6);

\path[-,font=\scriptsize]
(c5x1q-1) edge[-=0.8] node[pos=0.2,left,black] {} node[pos=0.8,left,black]{} (4.5,-3.6);

\path[-,font=\scriptsize]
(c5x01) edge[-=0.8] node[pos=0.2,left,black] {} node[pos=0.8,left,black]{} (4.5,-4.2);
   
\node at (4.65,0.5) {$\ldots$};
\node at (4.65,1.1) {$\ldots$};
\node at (4.65,-.5) {$\ldots$};
\node at (4.65,-1.1) {$\ldots$};

\node at (4,-.1) {$\vdots$};

\node at (4.65,-2) {$\ldots$};
\node at (4.65,-2.6) {$\ldots$};
\node at (4.65,-3.6) {$\ldots$};
    \node at (4.65,-4.2) {$\ldots$};

    
\vertex[circle,fill,minimum size = 5pt, label={below:\small $c_{2x[1:a]}'$}](c2x1a) at (-1,-2.5) {};
\vertex[circle,fill,minimum size = 5pt, label={below:\small $c_{x[1:a^{-1}]}'$}](cx1a-1) at (0,-2.5) {};

\path[-,font=\scriptsize]
   (c2x1a) edge[-=0.8] node[pos=0.2,auto,black]{\tiny $q^3$} 
    node[pos=0.8,above,black]{\tiny $q^2$}   (cx1a-1) ;

    \path[-,font=\scriptsize]
(1.4,-1.8)  edge[->=1,dashed, bend right] node[pos=0.5,above,black] {$a \in \Fq$}  (-.5,-2.3);

   \end{tikzpicture}
 \endpgfgraphicnamed
\]
\end{minipage} 
 \caption{The graph $\mathcal{G}_{\Phi_y',K}$ for $|y|=3$.}
  \label{fig9}
\end{figure}
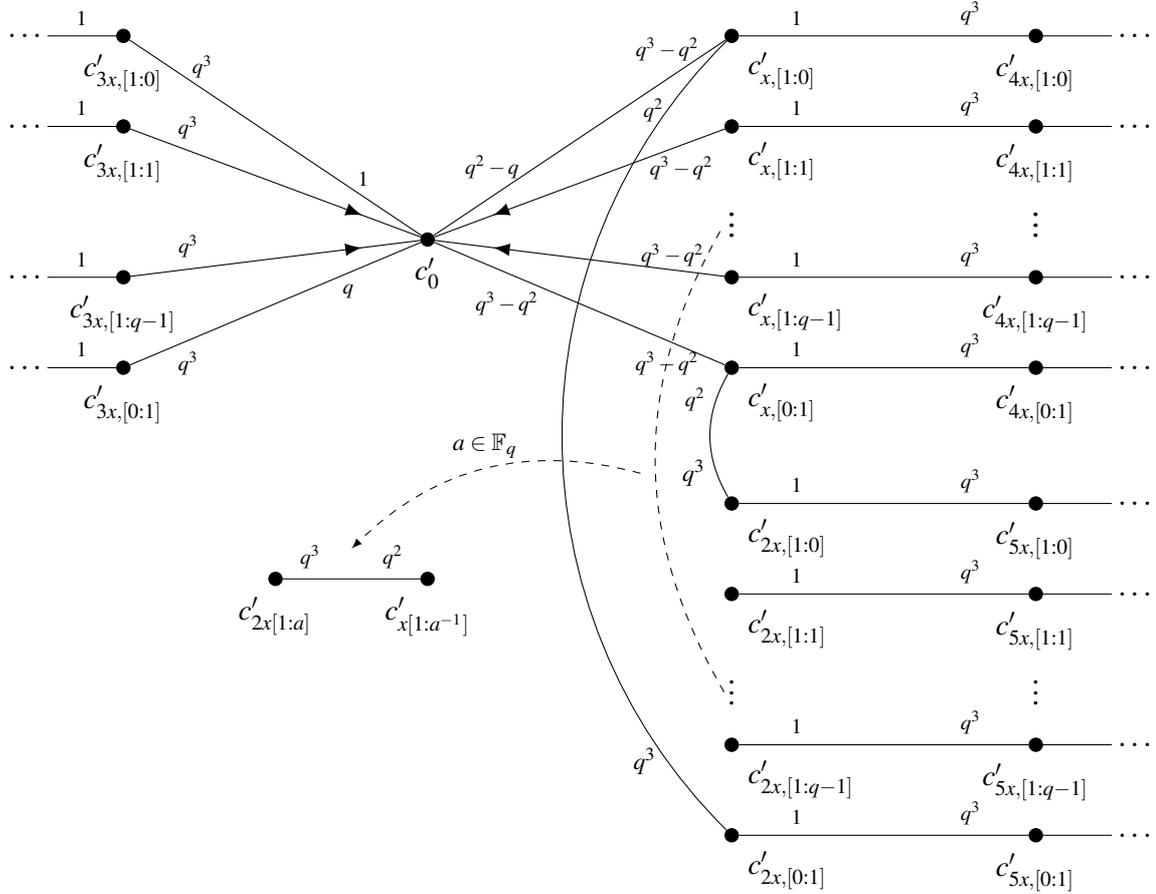

\end{center}


\section{General ramification in one place} \label{sec-generalramification}

Throughout this section, we fix $y$ a place of $\P^1$ of degree $d$. The goal here is to investigate the general situation when $K' \in \cV$ is given by some ramification at $y$ i.e.\ $K_z'=K_z$ for every places $z \neq y$. Hence, there are two important parameters to characterize such a $K'$. The first is the degree of the place $y$ and the second is the open subgroup $K_y' \leqslant  K_y$. We recall that $K_y=\GL_2(\FF_q[[\pi_y]])$. 

If we want to  exhaustively treat the problem of ramification in one place, we need either a method with very general arguments, or a classification of the open subgroups of $K_y$. We do not provide here any of these two things. Nonetheless, we consider an important family of open subgroups of $K_y$, the balls centered at $\id$ for the $\pi_y$-adic topology. Namely, we consider in this section $K_y'$ given by
\[ K_y^{(r)}:=\big\{ k \in K_y \;|\; k \equiv \id~\textrm{mod}(\pi_y^r)\big\}=GL_2(\kappa(y))+\pi_yM_2(\kappa_y[[\pi_y]])\]
for some positive integer $r$.
They are all open subgroups of $K_y$ and form a basis of compact neighborhoods of $\id$ for the $\pi_y$-adic topology. Thus, studying them is sufficient to understand the general problem of ramification in one place.

Therefore, the $K' \in \cV$ under consideration in this section is given by 
\[K' := K_y^{(r)}\times \prod_{z\neq y}K_z.\] 
In general, we shall denote these subgroups of $K$ by $K^{(y,r)}$, where $y$ represents the place where we consider the ramification and $r$ is the weight of ramification. Since we are fixing $y$ the place of ramification, we shall denote $K' =: K^{(y,r)}$ simply by $K^{(r)}$ whether there is no possible confusion.

The goal of this question is to investigate the action of $\cH_{K'}$ on $\cA^{K'}$ for $K'=K^{(r)}$. Observe that we have previously considered and solved this problem for $|y|=r=1$.

 \subsection*{Representation of the double coset} 
 In order to probe the action of $\cH_{K^{(r)}}$ on $\cA^{K^{(r)}}$, the first and important step is to find a system of representatives of the double coset $G(F) \setminus G(\mathbb{A})/Z K^{(r)}$.
There are different possibilities to get a system of representatives that generalize what we did in the last section. We show two natural strategies to do that. First, we need to introduce some notations. By definition, 
\[ \cdots \hookrightarrow K^{(r+1)}\hookrightarrow K^{(r)} \hookrightarrow \cdots \hookrightarrow K^{(1)} \hookrightarrow K^{(0)}:=K.\] 
And thus,
\[ \cdots \twoheadrightarrow G(F) \setminus G(\mathbb{A})/Z K^{(r+1)}\twoheadrightarrow G(F) \setminus G(\mathbb{A})/Z K^{(r)} \twoheadrightarrow \cdots \twoheadrightarrow G(F) \setminus G(\mathbb{A})/Z K^{(0)}\]

For every $i\geq j$, we can define
\[ P_{i,j} : G(F) \setminus G(\mathbb{A})/Z K^{(i)} \twoheadrightarrow G(F) \setminus G(\mathbb{A})/Z K^{(j)}\]
and denote $P_r$ for $P_{r,0}$. Observe that $P_{i,k}=P_{i,j}\circ P_{j,k}$ for every $i \geq j \geq k.$
In the following we find systems of representatives of $G(F) \setminus G(\mathbb{A})/Z K^{(r)}$ which are coherent with the surjections $P_{i,j}$  and generalize what we did in the last section. 

\begin{rem}\label{embeddings4.1}
The first step to achieve what we planned, would be to understand, for $i \geq j$, the quotient $K_{i,j}:= K^{(j)}/K^{(i)}$. First, remark that since $K^{(j)}$ and $K^{(i)}$ coincide on every other places different from $y$, it is enough to understand the quotient $K^{(j)}_y/K^{(i)}_y$. Let us remember that $(K^{(m)}_y)_m$ is not a general filtration of $\{\id\}$, indeed these groups are defined algebraically as kernels of reductions modulo powers of $\pi_y$. 

From the definition of $K^{(i)}_y$, we already have an answer for the case $j=0$. More precisely, $K_y/K^{(i)}_y$ is in bijection with the image of the reduction modulo $\pi_y^{i}$, which is exactly $\GL_2(\kappa(y))+\sum_{k=1}^{i-1}\pi_y^{i}M_2(\kappa(y))$. To conclude the case $j=0$, we observe that $\GL_2(\kappa(y))+\sum_{k=1}^{i-1}\pi_y^{i}M_2(\kappa(y))$ is in bijection with $(K_y-K_y^{i})\cup\{\id\}$. 

Previous construction extends for $j\neq0$, i.e. the restriction to $(K_y^j-K_y^{i})\cup\{\id\}$ of the canonical surjection 
\[ K^{(j)}_y \longrightarrow K^{(j)}_y/K^{(i)}_y\]
is, indeed,  a bijection. Since $K_y^{i}$ is sent to $\{\id\}$, we have the surjection. The injectivity comes from reducing modulo $\pi_y^{i}$ and applying the case $j=0$. All in one, we have a nice lift $K_{i,j}$ in $K_j$ and then it makes sense to write $K^{(j)}=K^{(i)}K_{i,j}$.

\end{rem}

\subsection*{First strategy to represents the double coset.} 
    Suppose $S$ is a system of representatives of $G(F)\setminus G(\mathbb{A})/ZK$.
By induction, we shall build $S_r$ a system of representatives of $G(F) \setminus G(\mathbb{A})/Z K^{(r)}$ in $G(\mathbb A)$.

We define $S_{0}=S$. Let $r \geq 1$ be an integer and $g \in S$. Recall that, according to Remark \ref{embeddings4.1},  $K_{r+1,r}$ is embedded in $K^{(r+1)} \subset G(\mathbb A)$. Then we can proceed as in Lemma \ref{lemma3.3}, i.e. if we fix $s_r \in S_r$, we have $P_{r+1,r}^{-1}([s_r])=[s_rK_{r+1,r}]$. Note that in the previous equation, the first $[.]$ means the class in $G(F) \setminus G(\mathbb{A})/Z K^{(r)}$ but the second means the class in $G(F) \setminus G(\mathbb{A})/Z K^{(r+1)}$. Now we just define $S_{s_r,r+1}$ as a subset of $K_{r+1,r}$ which is such that $s_rS_{s_r,r+1}$ is a system of representatives of $P_{r+1}^{-1}$. It is equivalent to consider the equivalence relations on $K_{r+1,r}$ induced by $G(F) \setminus s_rK_{r+1,r}/ZK^{(r)}$ and take a system of representatives $S_{s_r,r+1}$. Thus, the last step is to define $S_{r+1}=\coprod_{s_r \in S_r}s_rS_{s_r,r+1}$.

Hence, we showed a way to build iteratively a system of representatives of $G(F) \setminus G(\mathbb{A})/ZK^{(r)}$. At each step, we need to find a system of representatives for an equivalence relation on $M_2(\kappa(y))=\kappa(y)^4=\FF_{q^{|y|}}^4$. Elements $s_r \in S_r$ are given by construction, as following 
\[ s_r=sg(1+\pi_yM_1)...(1+\pi_y^rM_r)\] 
with $s \in S$, $g \in \GL_2(\kappa(y))$ and $M_i \in M_2(\kappa(y))$. Next, we define, 
\begin{eqnarray*}
   \theta_{i,j}:  S_r  & \longrightarrow &  G(\mathbb{A})  \\
    sg(1+\pi_yM_1)...(1+\pi_y^rM_i)  & \longmapsto &  sg(1+\pi_yM_1)...(1+\pi_y^rM_j) 
\end{eqnarray*}
for every $i \geq j$.
Here the meaning we give to the fact that the $S_r$ are coherent with the canonical surjections is that the application $\theta_{i,j} : S_i \longrightarrow S_j$ coincides with $P_{i,j}$ for every $i \geq j$.

\subsection*{Second strategy to represents the double cosets}

    Let $S$ be again a system of representatives of $G(F)\setminus G(\mathbb{A})/ZK$. Let $g \in S$ and $r \geq 1$. Let us see another way to build systems of representatives $S_r$ of the double cosets $G(F)\setminus G(\mathbb{A})/ZK^{(r)}$ that are coherent with the canonical surjections given before, which we will explain more specifically. The construction is similar to the previous one, but we will see that it is not the same.

    Recall that, according to remark \ref{embeddings4.1},  $K_{r,0}$ is embedded in $K^{(r)} \subset G(\mathbb A)$. Then we can do as in \ref{lemma3.3}, i.e. if we fix $s \in S$, we have $P_{r}^{-1}([s])=[sK_{r,0}]$. Note that in the previous equation, the first $[.]$ means the class in $G(F) \setminus G(\mathbb{A})/Z K^{(r)}$, but the second means the class in $G(F) \setminus G(\mathbb{A})/Z K^{(r+1)}$.

Now we fix $r\geq 1$ and let $s \in S$. We can take a subset $S_{s,r}$ of $K_{r,0}$ such that $s S_r$ is a system or representatives of $G(F)\setminus sK_{r,0}/ZK^{(r)}$.
Therefore, we define 
\[ S_r=\coprod_{s \in S}s S_{s,r}.\]
This strategy shows how we can get a system or representatives of $G(F)\setminus G(\mathbb{A})/ZK^{(r)}$ without knowing one for the $G(F)\setminus G(\mathbb{A})/ZK^{(k)}$, with $k \leq r-1$. Moreover, there is a natural way to deduce a system of representatives of $G(F)\setminus G(\mathbb{A})/ZK^{(k)}$ from a system of representatives of $G(F)\setminus G(\mathbb{A})/ZK^{(r)}$. Indeed, we verify that the application 
\begin{eqnarray*}
       \varphi_{r,k}: S_r & \longrightarrow & G(\mathbb{A}) \\
    s(g+\sum_{i=1}^{r-1}\pi_y^{i}M_i) & \longmapsto & s(g+\sum_{i=1}^{k-1}\pi_y^{i}M_i) 
\end{eqnarray*}
sends $S_r$ on a system or representatives of $G(\mathbb{A})/ZK^{(k)}$, for every $k \leq r$.

Let us now do the opposite and suppose that we already know such systems of representatives $S_k$ of $G(F)\setminus G(\mathbb{A})/ZK^{(k)}$, for every $k \leq r$ (or equivalently just one for $G(F)\setminus G(\mathbb{A})/ZK^{(r)}$).
Let $s \in S$, projecting $sK_{r+1,0}$ on $sS_{s,r}$ with $\varphi_{r+1,r}$, will make it possible for us to find a system of representatives of $sK_{r+1,0}$.
Let $s_r \in S_r$, yields 
\[ \varphi_{r+1,r}^{-1}(s_r)=\big\{s(s_r+\pi_y^rM) \;|\; M \in M_2(\kappa(y))\big\}.\] 
Let $S_{s,s_r,r+1}$ stand for a system of representatives for the equivalence relation that $G(F)\setminus \varphi_{r+1,r}^{-1}(s_r)/ZK^{(r+1)}$ induces on $M_2(\kappa(y))=\kappa(y)^4=\FF_{q^{|y|}}^4$. Then
\[S_{s,r}:=\coprod_{s_r \in S_r}s(s_r+\pi_y^rS_{s,s_r,r+1})\] 
is a system of representatives of $sK_{r+1,0}$.
And finally we get the system of representatives $S_{r+1}=\coprod_{s \in S}sS_{s,r+1}$.

Therefore, we get something very similar to the first strategy. We iteratively built the $S_r$ for each step where an equivalence relation on $M_2(\kappa(y))=\FF_{q^{|y|}}^4$ must be considered.
Here, the fact that $S_r$ is coherent with the canonical surjections means that the application $\varphi_{i,j}$ coincides with $P_{i,j}$ for every $i \geq j$.

\begin{rem} In both strategies above, we could consider $S =\{c_{nx} \;|\;n \in \N\}$ with $x \in |\P^1|$ a degree one place, as in the Section \ref{sec-unramified}. 
\end{rem}

Let us now summarize what we have just described in the previous two strategies. In both cases, we saw a technique to build systems of representatives $S_r$ of $G(F)\setminus G(\mathbb{A})/K^{(r)}$ by iterating $r$ steps. This is done by finding one after another systems, or representatives, of $G(F)\setminus G(\mathbb{A})/Z K^{(k)}$, for $k\leq r$. In every step we have to understand some equivalence relations on $\kappa(y)^4=\FF_{q^{d}}^4$, for $d = |y|$. The first strategy is more coherent with the multiplicative structure of $K$ and the second one more coherent with its additive structure. We can see that the systems or representatives that we got with these two techniques are not the same for $d>2$. Indeed, since one technique uses summation and the other uses multiplication, the degrees are different.

The two constructions that we gave before are not the only natural strategies to be considered. We can give an analogous construction for every partition of the integer $r$. For instance, the first strategy corresponds to $r=1+ \cdots +1$, whereas the second one corresponds to $r=r$. Again, by looking at the degree in $\pi_y$, we see that in general, all those systems of representatives are different.
For $K^{(1)}$ and $|y|=1$, there is only one strategy since $1=1$ is the only partition of the unity. In this case, the equivalence relation on $\FF_q^4$ was $(a,b,c,d)\sim(a',b',c',d')$ if and only if $cd'=dc'$, which gives $\P^1(\FF_q)$ for the quotient space, as we saw.

\subsection*{Similarities between cases $(1,2)$ and $(2,1)$}
 As we have seen before, fixed a positive integer $r$, $K^{(y,r)}$ only depends on the degree of $y$. Therefore we shall denote  $K^{(y,r)}$ simply by  $K^{(d,r)}$, where $d= |y|$.  Next, we investigate the similarities between 
 $K^{(1,2)}$ and $K^{(2,1)}.$
 
 Let us fix $x, y \in |\P^1|$ with $|x|=1$ and $|y|=2$ such that $ K^{(1,2)} := K^{(x,2)}$ and  $K^{(2,1)} := K^{(y,1)}$.  We aim to find a system of representatives of $G(F)\setminus G(\mathbb{A})/ZK^{(x,2)}$ and $G(F)\setminus G(\mathbb{A})/ZK^{(y,1)}$.
 Consider $S:=\{p_{nx} \in G(\A) \;|\; n \in \N\}$ as in the previous section. 

In the case $(1,2)$, according to the previous discussion, we need to understand when we have the equivalences 
\[ (p_{nx}(g_0+\pi_xM_1))_x\sim (p_{nx}(g_0'+\pi_xM_1'))_x \]
in $G(F) \setminus G(\mathbb{A})/ZK^{(x,2)}$,
for every $n \in \N$, $g_0,g_0' \in \GL_2(\FF_q)$ and $M_1,M_1' \in M_2(\FF_q)$.

In the case $(2,1)$, we have to do as in the last section, but the difficulty is that the place of the ramification is not the place that we used to represent $G(F) \setminus G(\mathbb{A})/ZK$. Indeed, they are not of the same degree. Thus, we want to understand when we have
\[ (p_{nx})_x(g)_y\sim (p_{nx})_x(g')_y\]
in $G(F) \setminus G(\mathbb{A})/ZK^{(y,1)}$,
for every $n \in \N$ and $g,g' \in  \GL_2(\kappa(y))=\GL_2(\FF_{q^2})$.

As in the previous section, namely by Lemma \ref{lemma-uni}, we can suppose that $\pi_x^2=\pi_y^{-1}$. Then $\{1, \pi_x^{-1}\}$ is a $\FF_q$-basis of $\kappa(y)=\FF_{q^2}$. Thus if $g=g_0+\pi_x^{-1}g_1$, yields
\[(p_{nx})_x(g)_y\sim((g_0+\pi_x^{-1}g_1)^{-1}p_{nx})_x.\]

Here, we would like to process as in previous section and reduce the problem to one single place. To do that, we generalize Proposition \ref{prop-ramifstrongaprox}.

\begin{prop}
 Let $x \in |\P^1|$ of degree one  and $K'=\prod_{y \in |\P^1|}K_y'\in \cV$. The map 
   \[ (.)_x : \Gamma_x'
\setminus G_x/Z_xK_x' \longrightarrow G(F) \setminus G(\mathbb{A})/ZK'\]
    with $\Gamma_x'= \bigcap_{y\neq x}(K_y' \cap G(G) )$ is well defined and injective.
\end{prop}

\begin{proof}
The proof of Proposition \ref{prop-ramifstrongaprox}, which shows that 
\[G(F) \setminus G(\mathbb{A})/ZK'=G(F) \setminus G(\mathbb{A})/Z_{x}K'\] still holds.
The fact that $G(F) \setminus G(\mathbb{A})/Z_{x}K'=\Gamma_x' \setminus G_x/Z_xK_x'$ is very similar.
Suppose that $a=z\gamma bk_x'$ with $a,b \in G_x$, $\gamma \in \Gamma_x'$ and $k_x' \in K_x'$. Define $k' \in K'$ such that $k'$ is equal to $\gamma^{-1}$ on every place $y\neq x$, and equal to $k_x'$ on the place $x$. Then we have $(a)_x=(z)_x\gamma (b)_xk'$ with $\gamma \in G(F)$ and $k' \in K'$. Reciprocally, suppose now that we have $(a)_x=(z)_xg(b)_xk'$ with $a,b \in G_x$, $g \in G(F)$ and $k' \in K'$. At the place $x$, this relation gives $a=zgbk_x'$, and  at every other place $y \neq x$, gives $gk_y'=1$. Thus $g \in \bigcap_{y\neq x}K_y'$, which shows that $g \in \Gamma_x'$ and concludes the proof.
\end{proof}

Previous proposition in the case of interest on the section yields the following. 

\begin{cor}
The maps
\[ (.)_x : \Gamma_x \setminus G_x/Z_xK_x^{(x,2)} \longrightarrow G(F) \setminus G(\mathbb{A})/ZK^{(x,2)}\]
and
\[ (.)_x : \Gamma_x^{(2)} \setminus G_x/Z_xK_x \longrightarrow G(F) \setminus G(\mathbb{A})/ZK^{(y,1)}\]
where $\Gamma_x^{(2)}:=\Gamma_x\cap K_y^{(x,2)}=
\big\{\gamma \in \Gamma_x \; | \; \gamma \equiv \id 
 \; \mathrm{mod}(\pi_x^{-2}) \big\}$ are well-defined and injective.

 $\hfill \square$
\end{cor}

Hence, the two problems pointed out above are left to understand when 
\[p_{nx}(g_0+\pi_xM_1)\sim p_{nx}(g_0'+\pi_xM_1') \] 
in $\Gamma_x \setminus G_x/Z_xK_x^{(x,2)}$,
for every $n \in \N$, $g_0,g_0' \in \GL_2(\FF_q)$ and $M_1,M_1' \in M_2(\FF_q)$. And when
\[(g_0+\pi_x^{-1}g_1)^{-1}p_{nx}\sim (g_0'+\pi_x^{-1}g_1')^{-1}p_{nx}\]
in $\Gamma_x^{(2)} \setminus G_x/Z_xK_x$,
for every $n \in \N$ and $g_0,g_0',g_1,g_1' \in M_2(\FF_q)$ such that $g_0+\pi_x^{-1}g_1,g_0'+\pi_x^{-1}g_1' \in \GL_2(\kappa(y))$.

\begin{prop} 
Let $n \in \N$, $g_0,g_0' \in \GL_2(\FF_q)$ and $g_1,g_1' \in M_2(\FF_q)$ such that $g_0+\pi_x^{-1}g_1,g_0'+\pi_x^{-1}g_1' \in \GL_2(\kappa(y))$ and $\det(g_0+\pi_xg_1)/\det(g_0'+\pi_xg_1') \in \FF_q^{\times}$. Then 
\[ p_{nx}(g_0+\pi_xg_1)\sim p_{nx}(g_0'+\pi_xg_1') \]
in $\Gamma_x \setminus G_x/Z_xK_x^{(x,2)}$, if and only if
\[(g_0^{-1}+\pi_x^{-1}g_1)p_{nx}\sim (g_0'+\pi_x^{-1}g_1')^{-1}p_{nx}\] 
in  $\Gamma_x^{(2)} \setminus G_x/Z_xK_x$
\end{prop}

\begin{proof}
  Suppose that $p_{nx}(g_0+\pi_xg_1)\sim p_{nx}(g_0'+\pi_xg_1')~\mbox{in } \Gamma_x \setminus G_x/Z_xK_x^{(x,2)}$.
  Then there exists $\gamma \in \Gamma_x$, $k \in K_x^{(x,2)}$ and $z \in Z_x$ such that
  $$z_x(g_0+\pi_xg_1)^{-1}p_{nx}^{-1}\gamma p_{nx}(g_0+\pi_xg_1)=k$$
  We denote $\gamma=\gamma(\pi_x^{-1})$ and $k=k(\pi_x)$. Since we want to prove an equivalence in $\Gamma_x^{(2)} \setminus G_x/Z_xK_x$, we can assume $z_x=1$.
Replacing $\pi_x $ by $\pi_x^{-1}$ yields
\[(g_0+\pi_x^{-1}g_1)^{-1}p_{nx}\gamma(\pi_x)p_{nx}^{-1}(g_0+\pi_x^{-1}g_1)=k(\pi_x^{-1}).\]
 By definition $\gamma(\pi_x) \in K_x$ and we would like to conclude $k(\pi_x^{-1}) \in \Gamma_x^{(2)}$. According to the last relation, every matrices on the left have entries with finite degree on $\pi_x$, except maybe $\det(g_0+\pi_xg_1)^{-1}$ which is in $Z_x$. Thus, we can suppose that the coefficients of $k(\pi_x^{-1})$ are in $\FF_q[\pi_x^{-1}]$, module $(Z_x)$. Finally, by hypothesis $\det(k(\pi_x^{-1})) \in \FF_q^{\times} =(\Gamma_x^{(2)})^{\times}$.
Therefore, 
 \[(g_0+\pi_x^{-1}g_1)^{-1}p_{nx}=z_xk(\pi_x^{-1})(g_0'+\pi_x^{-1}g_1')^{-1}p_{nx}\gamma(\pi_x)^{-1}\]
 with $k(\pi_x^{-1}) \in \Gamma_x^{(2)} $ and $\gamma(\pi_x)^{-1} \in K_x$, which concludes the first implication.
The reciprocal implication is analogous.
\end{proof}

\begin{rem}
    By considering $g_0,g_0' \in \GL_2(\FF_q)$ and $g_1,g_1' \in M_2(\FF_q)$ such that $g_0+\pi_x^{-1}g_1,g_0'+\pi_x^{-1}g_1' \in \GL_2(\kappa(y))$ and $\det(g_0+\pi_xg_1)/\det(g_0'+\pi_xg_1') \in \FF_q^{\times}$,  the previous proposition shows an explicit identification between a subset of $G_x/Z_xK_x^{(x,2)}$ and a subset of $\Gamma_x^{(2)} \setminus G_x/Z_xK_x$. Moreover, it induces an explicit identification between some subsets of $G(F) \setminus G(\mathbb{A})/ZK^{(x,2)}$ and $G(F) \setminus G(\mathbb{A})/ZK^{(y,1)}$.
    The hypothesis is very restrictive here, but we guess that the result is true even with a weaker hypothesis. 
    We also expected that one can get analogous and more general result. By adopting the same reasoning, it would seem worth investigating the similarities between the cases $(d,r)$ and $(1,rd)$.
\end{rem}

\begin{rem} The usefulness of the previous proposition is as follows. 
By projecting on $\Gamma_x \setminus G_x/Z_xK_x^{(x,2)}$ we remark that
\[p_{nx}(g_0+\pi_xg_1)\sim p_{nx}(g_0'+\pi_xg_1')\]
in $\Gamma_x \setminus G_x/Z_xK_x^{(x,2)}$  implies  $p_{nx}g_0\sim p_{nx}g_0'$
in $\Gamma_x \setminus G_x/Z_xK_x^{(x,1)}.$
But the second equivalence is what we did in the previous section.
Thus, the problem become of knowing when 
\[p_{nx}(\vartheta_w+\pi_xg)\sim p_{nx}(\vartheta_w+\pi_xg')\]
in $\Gamma_x \setminus G_x/Z_xK_x^{(x,2)}$
for every $n \in \N$, $w \in \P^1(\FF_q)$ and $g,g' \in M_2(\FF_q)$.
It induces an equivalence relation on $ M_2(\FF_q)=\FF_q^4$ and we are left to understand it.    
\end{rem}

\begin{rem} 
To summarize,  we can say that there are similarities between the cases $(1,2)$ and $(2,1)$, and there seems to be the same kind of similarities between $(d,r)$ and $(1, rd)$. With this established, we are left to treat the cases where $d=1$. This case is easier with what we saw because  we need to understand an equivalence relation on $\FF_{q^{|y|}}^4$ to go from a system of representatives of $G(F)\setminus G(\mathbb{A})/ZK^{(y,r)}$ to a system of representatives of $G(F)\setminus G(\mathbb{A})/ZK^{(y,r+1)}$, which seems easier when $|y|=1$.
\end{rem}

\subsection*{A result to simplify the case $(1,r)$}

Next, we give a result, which is a generalization of Proposition \ref{prop-ramifequiv}, for the case $(1,r)$. We fix $x \in |\P^1|$ a degree one place and denote $K^{(r)}$ instead of $K^{(x,r)}$ or $K^{(1,r)}$.

\begin{prop}
Let $n \in \N^{\times}$ and $g,g'\in \GL_2(\FF_q)$. Then
\[p_{nx}(\sum_{i=0}^{r-1}g_i\pi_x^i)\sim p_{nx}(\sum_{i=0}^{r-1}g_i'\pi_x^i) \; \textrm{ in } \;\Gamma_x \setminus G_x/Z_xK_x^{(r)}\]
if and only if 
\[ p_{nx}(\sum_{i=0}^{r-1}g_i\pi_x^i)\sim p_{nx}(\sum_{i=0}^{r-1}g_i'\pi_x^i) \textrm{ in }\Gamma_x \setminus G_x/Z_xK_{x,r}^{(r)}\]
with $K_{x,r}^{(r)}$ be the subset of $K_{x}^{(r)}$ whose matrices have polynomials of degree less than or equal to $2(r-1)$ in the entries.
\end{prop}

\begin{rem}
    Since $K_{x,r}^{(r)}$ is not a group, it is not clear so far that the second binary relation is an equivalent relation. The following proof then also shows that this binary relation is an equivalence relation.
\end{rem}

\begin{proof}
Since $K_{x,r}^{(r)} \subseteq K_{x}^{(r)}$ we only need to prove the direct implication.
We suppose that $p_{nx}(\sum_{i=0}^{r-1}g_i\pi_x^i)\sim p_{nx}(\sum_{i=0}^{r-1}g_i'\pi_x^i)$  in  $\Gamma_x \setminus G_x/K_x^{(r,1)}$. Thus, there are $\gamma \in \Gamma_x$, $z \in Z_x$ and $k \in K_x^{(r)}$ such that
$$p_{nx}(\sum_{i=0}^{r-1}g_i\pi_x^i)k(\sum_{i=0}^{d-1}(g_i')^{-1}\pi_x^i)p_{nx}^{-1}=\gamma.$$
Since we are considering an equivalence in $\Gamma_x \setminus G_x/Z_xK_{x,r}^{(r)}$, we can assume is $z=1$.

Moreover, since $(\sum_{i=0}^{r-1}g_i'\pi_x^i)^{-1}$ is equivalent to $\sum_{i=0}^{r-1}(g_i')^{-1}\pi_x^i$ module $Z_x$, yields
$$p_{nx}(\sum_{i=0}^{r-1}g_i\pi_x^i)k(\sum_{i=0}^{r-1}g_i'^{-1}\pi_x^i)p_{nx}^{-1} \in \Gamma_x.$$
Thus, we might write $k=k'+\pi_x^{2(r-1)}k''$ with $k' \in M_2((\FF_q)[\pi_x])$ and $k'' \in M_2(\FF_q[[\pi_x]])$, where the entries of $k'$ are polynomials with coefficients in $\FF_q$ and degree $\leq 2(r-1)$. Therefore
$$p_{nx}(\sum_{i=0}^{r-1}g_i\pi_x^i)(k'+\pi_x^{2r-1}k'')(\sum_{i=0}^{r-1}g_i'^{-1}\pi_x^i)p_{nx}^{-1} \in \Gamma_x.$$
Since the coefficients of $(\sum_{i=0}^{r-1}g_i\pi_x^i)(\sum_{i=0}^{r-1}g_i'^{-1}\pi_x^i)$ are of degree lesser or equal to $2(r-1)$ it implies, with the same technique we use in proof of Proposition \ref{prop-ramifequiv}, that
$$p_{nx}(\sum_{i=0}^{r-1}g_i\pi_x^i)k'(\sum_{i=0}^{r-1}g_i'^{-1}\pi_x^i)p_{nx}^{-1} \in \Gamma_x$$
with $k'\in K_{x,r}^{(r)}$. 
\end{proof}

\begin{rem}
    The usefulness of this result is that $K_{x,r}^{(r)}$ is a finite set, which was not the case for $K_{x}^{(r)}$. Then it can be useful to use an informatics program, for instance.
\end{rem}

We say nothing more here about representation of $G(F) \setminus G(\mathbb{A})/ZK^{(r)}$. We hope that what we just saw can be useful and a good starting point for an exhaustive investigation on the question. 
Obviously, it would not be finished yet. 
The question of ramifications in different places (a finite number of them) i.e. $K'$ with $K_y' \neq K_y$ for different places $y$ would remain.

\subsection*{More on graphs of ramified Hecke operators}
We fix $x \in |\P^1|$ of degree one. In the following, we 
generalize Proposition \ref{propdecomp} to $K^{(x,r)}$.

\begin{prop} Let $x \in |\P^1|$ of degree one. Then 
    \begin{enumerate}[(i)]
        \item $K^{(x,r)}p_{y}^{-1}K^{(x,r)}=\coprod_{w\in \P^1(\kappa(y))} \xi_{y,w} K^{(x,r)},$ for all $ y\neq x.$ \\
        \item $K^{(x,r)}p_x^{-1}K^{(x,r)}=\coprod_{a \in \FF_q}\begin{pmatrix} \pi_x&a\pi_x^{r-1} \\ &1 \end{pmatrix} K^{(x,r)}.$
    \end{enumerate}
\end{prop}

\begin{proof}
    The proof is exactly analogous to Proposition \ref{propdecomp}.
    The only difference is that in this case, 
    \[\begin{pmatrix} \pi_x&Q \\ &1 \end{pmatrix}K^{(x,r)}=\begin{pmatrix} \pi_x&R \\ &1 \end{pmatrix}K^{(x,r)}\]
    if and only if $\pi_x^r$ divides $Q-R$.
\end{proof}

\begin{rem}
    We state the above proposition for a degree one place because we guess that the similarities between $(d,r)$ and $(1,rd)$ still hold, not only for the vertices, but also for the computing of graphs.
   \end{rem}

    We consider  the Hecke operator $\Phi_y^{(x,r)}$ given by the characteristic function on 
    \[ K^{(x,r)}\begin{pmatrix} \pi_y &  \\ &1 \end{pmatrix} K^{(x,r)}\]
divided by the volume of $K^{(x,r)}.$ Besides, there are different Hecke operators in $\cH_{K^{(x,r)}}$, these are the generalization of the generators of the unramified Hecke algebra, cf. section \ref{sec-unramified}.
The following theorem is an analogous to Theorem \ref{thm-ramifiededges}.

\begin{thm} \label{thm-generalramifiedgraphs}
 Let $x,y \in |\P^1|$ with $|x|=1$ and $y \neq x$.
    \begin{enumerate}[(i)]
        \item The set of $\Phi_y^{(x,r)}$-neighbors of $g \in G(\A)$   is 
        \[\big\{g(\xi_{y,w})_y \;|\; w\in \P^1(\kappa(y))\big\},\] 
        Moreover, for every $w \in \P^1(\kappa(y))$ 
        \[ m(g,g\xi_{y,w})=\#\big\{w'\in \P^1(\kappa(y)) \;  | \; g \xi_{y,w'} \sim g\xi_{y,w} \}.\]
        Therefore, every point has $q^{|y|}+1$ $\Phi_y^{(x,r)}$-neighbors when counted with multiplicity.
        \item The set of $\Phi_y^{(x,r)}$-neighbors of $g \in G(\A)$  is 
        \[ \Big\{ g { \tiny \begin{pmatrix} \pi_x&a\pi_x \\ &1 \end{pmatrix} }\;|\; a\in \FF_q\Big\}.\] 
        Moreover, for every $a \in \FF_q$
        \[ m(g,g { \tiny \begin{pmatrix} \pi_x&a\pi_x^{d-1} \\ &1 \end{pmatrix}})=\#\Big\{a'\in \FF_q \;\big|\; g{ \tiny \begin{pmatrix} \pi_x&a'\pi_x \\ &1 \end{pmatrix}} \sim g { \tiny  \begin{pmatrix} \pi_x&a\pi_x \\ &1 \end{pmatrix}} \Big\}.\]
        Therefore, every point has $q$ \; $\Phi_y^{(x,r)}$-neighbors when counted with multiplicity.
    \end{enumerate}
   $ \hfill \square$
\end{thm}

\begin{rem}
   Analogous to Remark \ref{rem-calculation}, given a systems of representatives that are coherent with the applications $P_{i,j}$, we find the class of an adelic matrix in the double coset $G(F) \setminus G(\mathbb{A})/ZK^{(x,r)}$. In order to do that, we project consecutively on the double cosets $G(F) \setminus G(\mathbb{A})/ZK$, \ldots , $G(F) \setminus G(\mathbb{A})/ZK^{(x,r-1)}$.
\end{rem}

\subsection*{More on ramified Hecke eigenspaces}
We end this article with a couple of  conjectures about ramified Hecke eigenspaces. It is analogous to Theorems \ref{thm-mainramified1} and \ref{thm-mainramified2}.

We observe that in both unramified and ramified cases, what allowed us to compute the dimension of eigenspaces was the generic shape of the graph that we draw on Figure \ref{fig6} for the ramified case. We think that such a process can be generalized to $K^{(x,r)}$ for $x$ a degree one place. We fix in the following $x \in |P^1|$ be a degree one place. 
If we observe the graph of $\Phi_y$ or $\Phi_y'$ or $y \neq x$, we can see some ``branches'' going to infinite. Remark that the dimension of the correspondent eigenspace is always the number of such ``branches''.

In order to state the conjecture about the dimension of the ramified Hecke eigenspaces, we first define more precisely what we mean about the ``branches''.
Therefore, we first state a  conjecture which define the objects.

\begin{conj}
For every $r \in \N$, it exists $m \in\ N$ such that for every $n \geq m$, we have $|P_r^{-1}(c_{nx})|=|P_r^{-1}(c_{mx})|$, where $P_r$ are the canonical projection defined at the begging of this section. 
\end{conj}

Let us fix $r \in \N$. If the above conjecture holds, we can consider a minimal such $m \in\ N$, which we denote by $m_r$.
Define $C_r :=G(F) \setminus G(\mathbb{A})/ZK-\{c_0,c_x,...,c_{m_r x}\}$.

\begin{conj} \label{conj-dimension}
For every $r \in \N$ and $\lambda \in \C^{\times}$,  
\[\dim \cA^{K^{(x,r)}}(\Phi_y^{(x,r)}, \lambda)=|y|.|P_r^{-1}(C_r)/C_r|.\]
\end{conj}

\begin{rem}  
For $r=0$, i.e. in the unramified case, we have $m_0=0$ and $$|(G(F) \setminus G(\mathbb{A})/ZK)/(G(F) \setminus G(\mathbb{A})/ZK)|=1.$$ 
Then we find back Theorem \ref{thm-mainunramified}.
\end{rem}

\begin{rem}
For $r=1$, we have $m_1=1$ and 
$$|(G(F) \setminus G(\mathbb{A})/ZK^{(x,1)}-\{c_0'\})/(G(F) \setminus G(\mathbb{A})/ZK-\{c_0\})|=|\P^1(\FF_q)|=q+1.$$
Then we find back Theorem \ref{thm-mainramified2}.
\end{rem}

\section*{Acknowledgements} The first author was supported by FAPESP 
grant number 2022/09476-7. This research project took place in the context of an internship of the second author in ICMC-USP with the first author. Second author would like to thank the warm welcoming of ICMC-USP. Both authors thank the reviewer for the useful comments and suggestions.

\end{document}